\documentclass[
    12pt,       
    a4paper,
]{article}

\usepackage[margin=1in]{geometry}
\usepackage[onehalfspacing]{setspace}

\usepackage{siunitx}
\usepackage{amsthm}
\usepackage{amssymb, mathtools}
\usepackage{amsmath}
\usepackage{mleftright}
\usepackage{marvosym}

\usepackage{mathtools}
\usepackage{tikz}
\usetikzlibrary{calc,intersections,through,arrows.meta,external}
\usepackage{tkz-euclide}

\usepackage{placeins}

\usepackage[color=orange!25]{todonotes}
\presetkeys{todonotes}{size=\setstretch{1}}{} 
\usepackage{booktabs}
\usepackage{multirow}
\usepackage{threeparttable}

\usepackage{xspace}
\usepackage{stmaryrd}
\usepackage{datetime}
\usepackage{graphicx}
\usepackage[labelformat=simple]{subcaption}

\usepackage[american]{babel}
\usepackage{microtype}
\usepackage{spverbatim}
\usepackage[hidelinks,
            colorlinks = false,
            allcolors=blue]{hyperref} 
\usepackage{enumitem}
\usepackage{bm}
\usepackage{rotating}
\usepackage{cancel}

\usepackage{algorithm}
\usepackage{algpseudocodex}

\usepackage{nicefrac}
\usepackage[title]{appendix}

\usepackage{aligned-overset}

\usepackage{csquotes}
\usepackage[
    date=year,
    bibstyle=ext-numeric,
    citestyle=numeric-comp,
    sorting=nyt,
    defernumbers=true,
    sortcites=true,
    maxbibnames=9999,
    maxcitenames=2,
    giveninits=true,
    backend=biber
    ]{biblatex}
\stdpunctuation{}

\DeclareFieldFormat%
[article,inbook,incollection,inproceedings,patent,thesis,unpublished]
{titlecase:title}{\MakeSentenceCase*{#1}}

\DefineBibliographyExtras{british}{%
  \DeclareBibstringSet{latin}{andothers,ibidem}%
  \DeclareBibstringSetFormat{latin}{\mkbibemph{#1}}%
}
\UndefineBibliographyExtras{british}{%
  \UndeclareBibstringSet{latin}%
}

\theoremstyle{thmstyleone}%
\newtheorem{theorem}{Theorem}
\newtheorem{proposition}[theorem]{Proposition}%
\newtheorem{lemma}[theorem]{Lemma}%
\newtheorem{corollary}[theorem]{Corollary}%

\theoremstyle{thmstyletwo}%
\newtheorem{example}{Example}%
\newtheorem{remark}{Remark}%

\theoremstyle{thmstylethree}%
\newtheorem{definition}{Definition}%
\newtheorem{assumption}{Assumption}%

\graphicspath{{figures/}}

\newcommand{\T}{\mathsf{T}}

\newcolumntype{C}{>{{}}c<{{}}}

\newcommand{\appri}{\mathcal{A}}
\newcommand{\lowi}{\appri^*}
\newcommand{\lowieps}{\appri^{1+\varepsilon}}
\newcommand{\lowiepsarg}[1]{\appri^{#1}}

\newcommand{\subdi}{\mathcal{G}}
\newcommand{\subdielem}{G}
\newcommand{\subdivert}{g}
\newcommand{\coneHull}{W}
\newcommand{\lwfunc}{\mathcal{L}}

\newcommand{\conv}{\mathsf{conv}}
\newcommand{\cone}{\mathsf{cone}}
\newcommand{\vertx}{\mathsf{vert}}

\newcommand{\bigO}{\mathcal{O}}
\newcommand{\poly}{\mathrm{poly}}

\newcommand{\optset}{S}
\newcommand{\oracle}{\mathsf{Alg}}

\newcommand{\Vpoly}{$\mathcal{V}$-representation}
\newcommand{\Hpoly}{$\mathcal{H}$-representation}

\newcommand{\orsols}{X_{\textsf{pol}}}

\newcommand{\novert}{\nu_v}
\newcommand{\nofaces}{\nu_f}
\newcommand{\nofacet}{\nu_F}

\newcommand{\multiSingleBase}{\mathrm{MOP}^1}
\newcommand{\multiBiBase}{\mathrm{MOP}^2}
\newcommand{\paraBiBase}{\Pi^2}

\newcommand{\nonApxArea}{\Delta}
\newcommand{\Xkone}{X^1}
\newcommand{\Ykone}{Y^1}

\newcommand{\YEkone}{Y_{\mathsf{E}}^1}
\newcommand{\Xktwo}{X^2}
\newcommand{\Yktwo}{Y^2}
\newcommand{\XEktwo}{X_{\mathsf{E}}^2}
\newcommand{\YEktwo}{Y_{\mathsf{E}}^2}

\newcommand{\vol}{\mathsf{vol}}

\newcommand{\Smax}{S_{\max}}
\newcommand{\smax}{s_{\max}}

\renewcommand{\det}{\mathsf{det}}

\begin{document}

\title{An Adaptive Algorithm for the Approximation of General Linear-Parametric Optimization Problems}

\author{Levin Nemesch\textsuperscript{\Letter}\footnote{RPTU Kaiserslautern-Landau, Department of Mathematics, Paul-Ehrlich-Str. 31, 67663 Kaiserslautern, Germany, email:\{l.nemesch,stefan.ruzika\}@math.rptu.de}, Stefan Ruzika\footnotemark[\value{footnote}]~,
     Clemens Thielen\footnote{Technical University of Munich, Campus Straubing for Biotechnology and Sustainability, Professorship of Optimization and Sustainable Decision Making, Am Essigberg~3, 94315~Straubing, Germany, email:\{clemens.thielen,alina.wittmann\}@tum.de}, Alina Wittmann\footnotemark[\value{footnote}]
}


\maketitle

\begin{abstract}

\singlespacing\small

Linear-multi-parametric optimization problems are a widely studied class of optimization problems.
The objective function in such a problem is affine linear dependent on a parameter vector, and the goal is to compute a set of solutions that contains an optimal solution for every fixed parameter vector.
However, this is known to be computationally challenging: The underlying non-parametric problem might be NP-hard, and, in addition, optimal solution sets might have exponential cardinality.
Parametric approximation aims at providing polynomial-time algorithms that overcome these challenges.
Instead of computing an optimal solution set, the goal is to compute an approximation set that contains only an approximate solution for every fixed parameter vector.
Several new parametric approximation algorithms have been developed in recent literature.
However, all of these share a common set of assumptions, which limits the class of parametric optimization problems that can be approximated.
Namely, they do not allow negative parameter dependencies and have their parameter sets fixed to the positive orthant.
We present a new adaptive approximation (and, also, exact) algorithm that can be applied to a wider class of linear-multi-parametric optimization problems.
Our algorithm builds upon existing algorithms from both the fields of parametric and multi-objective optimization and generalizes these algorithms.
In addition, we provide structural results for the transformation of parameter sets, and demonstrate that, for linear-multi-parametric maximization problems, the assumption of non-negative optimal objective values over the whole parameter set is not sufficient to ensure approximability.%

\medskip

\noindent
\textbf{Funding:} 
This research was funded by the Deutsche Forschungsgemeinschaft (DFG, German Research Foundation) --- Project number 508981269 and GRK 2982, 516090167 ``Mathematics of Interdisciplinary Multiobjective Optimization''.

\end{abstract}

\clearpage

\section{Introduction}

When modelling real-life problems as optimization problems, a common challenge arises from the uncertainty of \emph{parameters} in the model.
Coefficients affecting the objective function or the feasible region may not be known precisely at the time of decision making, and they might vary over time.
Developing concepts and algorithms to approach such problems is the subject of \emph{parametric optimization}.
In some cases, one is interested in the set of parameters for which a given solution stays optimal or feasible.
In other cases, the goal is to determine a complete set of solutions such that, for each possible parameter vector, at least one solution in this set is optimal.

In this article, we look at the latter case: finding (optimal) solution sets for so-called \emph{linear-multi-parametric optimization problems}.
The class of these problems is widely studied and has many applications (see, e.g.,~\cite{nemesch.ruzika.ea2025SurveyOfExact}).
However, many linear-multi-parametric optimization problems are known to be computationally hard.
Even for a single, fixed parameter vector, determining an optimal solution can be NP-hard.
Additionally, the cardinality of an optimal solution set that covers all possible parameter vectors might be exponential in the input size~\cite{murty1980ComputationalComplexity,carstensen1983ComplexityParametric,ruhe1988ComplexityResults}.
For both these reasons, we focus on the \emph{approximation} of linear-multi-parametric optimization problems.
Approximating an instance of a linear-multi-parametric optimization problem is understood as finding a set of solution where, for each possible parameter vector, the problem is approximated by at least one solution from this set.
This notion of approximation overcomes both the problem of NP-hardness and of solution sets of exponential cardinality --- for many linear-multi-parametric optimization problems, approximation sets (of polynomial cardinality) can be computed in polynomial time~\cite{helfrich.herzel.ea2022ApproximationAlgorithmGeneral}.
We present an algorithm that can act both as an exact and as an approximation algorithm for linear-multi-parametric optimization problems.

Next, we formally define the concept of linear-multi-parametric optimization and establish some important notions which are essential for our approach.

\begin{definition}\label{def::paraproblem}
    Let $p\in\mathbb{N}_{>0}$.
    A \emph{linear-$p$-parametric optimization problem} (or \emph{linear-(multi-) parametric optimization problem}) is a set of instances.
    An instance~$\Pi$ of a linear-$p$-parametric optimization problem is a triple $\Pi=(\Lambda, X,f)$, where:
    \begin{itemize}
        \item $\Lambda=\{\lambda\in\mathbb{R}^p: Q\lambda\geq d\}$ for some $m\in\mathbb{N}_{>0}$, $Q\in\mathbb{Q}^{m\times p},d\in\mathbb{Q}^{m}$,
        \item $X$ is a set, and
        \item $f: X\to \mathbb{R}^{p+1}, x\mapsto {\left(f_1(x),\ldots, f_{p+1}(x)\right)}^\T$.
    \end{itemize}
    The goal is to solve the problem
    \[
    {\left\{\min_{x\in X} \ (\text{or } \max_{x\in X}) \quad F_\lambda(x) \coloneq \sum_{i=1}^p \lambda_i f_i(x) + f_{p+1}(x) \right\}}_{\lambda\in\Lambda}.
    \]
\end{definition}
The set~$\Lambda$ is called the \emph{parameter set}, the set~$X$ is called the \emph{set of feasible solutions}, the function~$f$ is called the \emph{objective function}, and we call the function~$F_\lambda \colon X \rightarrow \mathbb{R}$ the \emph{parametric value function}.
If $\lambda\in\Lambda$ is fixed to a specific value, we obtain an instance of a (non-parametric) optimization problem, denoted by $\Pi(\lambda)$.
As a shorthand, $F^*(\lambda)\coloneq \min_{x\in X}F_{\lambda}(x)$ (or $\max_{x\in X}F_{\lambda}(x)$) denotes the optimal objective value of $\Pi(\lambda)$.
Solving~$\Pi$ requires finding an \emph{optimal solution set}, i.e., a set~$S\subseteq X$ of solutions such that, for each~$\lambda\in\Lambda$, at least one~$x\in S$ is an optimal solution of $\Pi(\lambda)$.
In an instance of a linear-multi-parametric optimization problem, only the encoding of~$\Lambda$ is predetermined (by the matrix~$Q$ and the vector~$d$).
The set~$X$ and the function~$f$ can be encoded implicitly.
For example, an instance of the \emph{linear-parametric shortest path problem} consists of some $Q,d$ for the parameter set~$\Lambda$, a graph with edge-cost vectors, and two nodes~$s,t$.
The set~$X$ is the set of all (simple) $s$-$t$ paths, and~$f$ sums up the cost vectors of all edges in a path.

The term \emph{linear-parametric} indicates that the parameter dependency is linear and occurs only in the parametric value function.
We additionally assume that the parameter set is polyhedral.
This distinguishes linear-parametric optimization problems from more general parametric optimization problems, where parameter dependencies can be non-linear and also occur in the set of feasible solutions, and where parameter sets can be arbitrary.
The linear-parametric setting is quite common in literature~\cite{nemesch.ruzika.ea2025SurveyOfExact}.
For the rest of this paper, we simplify \emph{linear-(multi)-parametric} into \emph{parametric} for easier notation, and only write \emph{$p$-parametric} to emphasize the parameter dimension.

If~$F_\lambda$ is to be minimized, we call the problem a \emph{parametric minimization problem}, and if it is to be maximized, we call it a \emph{parametric maximization problem}.
For ease of description, we assume minimization throughout this paper and discuss the maximization case only where a generalization is not straight-forward.

\medskip

Even in the $1$-parametric case, it is well-known that a parametric optimization problems can be \emph{intractable}, i.e., there are families of instances where every optimal solution set has exponential cardinality in the encoding length of the underlying instance.
Examples include the parametric variants of the knapsack problem~\cite{carstensen1983ComplexityParametric}, the shortest path problem~\cite{carstensen1983ComplexityParametric}, parametric linear~\cite{murty1980ComputationalComplexity,ruhe1988ComplexityResults} and integer programming~\cite{carstensen1983ComplexityParametric}, and many more.

For $\varepsilon\geq0$ and a parametric minimization problem instance $\Pi$, a \emph{$(1+\varepsilon)$-approximation set} is a set $\optset\subseteq X$ where, for each $\lambda\in\Lambda$, there exists~$x\in\optset$ such that $F_\lambda(x)\leq (1+\varepsilon) \cdot F^*(\lambda)$.
We say that the solution~$x$ \emph{$(1+\varepsilon)$-approximates}~$\Pi(\lambda)$.
A \emph{$(1+\varepsilon)$-approximation algorithm} is an algorithm that, for a fixed $\varepsilon\geq0$, takes as input a parametric minimization problem instance~$\Pi$ and computes a $(1+\varepsilon)$-approximation set for~$\Pi$ in time $\bigO(\poly\langle \Pi\rangle)$, where $\langle \Pi\rangle$ denotes the binary encoding length of~$\Pi$, and $\poly()$ denotes an arbitrary polynomial function with bounded degree.
An algorithm that takes as input an instance~$\Pi$ and an~$\varepsilon>0$ is called a parametric \emph{polynomial-time approximation scheme}~(PTAS) if it computes a $(1+\varepsilon)$-approximation set for~$\Pi$ in time $\bigO(\poly\langle \Pi \rangle)$ for every fixed value of~$\varepsilon$. It is called a parametric \emph{fully polynomial-time approximation scheme}~(FPTAS) if it computes a $(1+\varepsilon)$-approximation set for $\Pi$ in time $\bigO(\poly(\langle \Pi,\varepsilon\rangle,\varepsilon^{-1}))$.

In contrast, for parametric maximization problems, we consider $\varepsilon$ such that $0\leq\varepsilon<1$.
A set $S\subseteq X$ is called an \emph{$(1-\varepsilon)$-approximation set} if, for every $\lambda\in\Lambda$, there is an $x\in S$ with $F_\lambda(x)\geq (1-\varepsilon) F^*(\lambda)$.
The concepts of \emph{$(1-\varepsilon)$-approximation algorithms} and (F)PTAS for maximization are defined analogously to the minimization case.

\medskip

Under rather mild assumptions, for every $\varepsilon>0$, there exists a $(1+\varepsilon)$-approximation set of polynomial cardinality for every parametric optimization problem.
We discuss these assumptions in Section~\ref{sec::prelims::assumptions}, and summarize them in Assumption~\ref{assumption::general}.

\medskip

To the best of our knowledge, \Citeauthor{katoh.ibaraki1987ParametricCharacterizationEapproximation}~\cite{katoh.ibaraki1987ParametricCharacterizationEapproximation} are the first to consider parametric approximation sets.
There, the parameter set~$\Lambda$ is partitioned into a grid.
For every vertex~$\lambda$ of this grid, the problem $\Pi(\lambda)$ needs to be solved (or at least approximated).
In recent years, interest in parametric approximation algorithms has continuously increased.
A general approximation algorithm for $1$-parametric optimization problems is developed by~\citeauthor{bazgan.herzel.ea2022ApproximationAlgorithmGeneral} in~\cite{bazgan.herzel.ea2022ApproximationAlgorithmGeneral}.
It generalizes the well-known Eisner-Severance method~\cite{eisner.severance1976MathematicalTechniques}.
Another grid-based approach independent of the one by~\citeauthor{katoh.ibaraki1987ParametricCharacterizationEapproximation} is developed by~\citeauthor{helfrich.herzel.ea2022ApproximationAlgorithmGeneral}~\cite{helfrich.herzel.ea2022ApproximationAlgorithmGeneral}.
Specialized approximation algorithms for parametric variants of the knapsack problem are developed in~\cite{giudici.halffmann2017ApproximationSchemes,holzhauser.krumke2017AnFPTAS,halman.holzhauser.ea2018FPTAKnapsackParaWeights}.

It is well-known that parametric optimization and multi-objective optimization are closely related~\cite{isermann1977EnumerationSetAll}, we outline this relation in more detail in Section~\ref{sec::prelims}.
A consequence of this relation is that approximation algorithms from multi-objective optimization can be used for parametric optimization, see~\cite{nemesch.ruzika.ea2025SurveyOfExact} for an overview.
In particular, the notion of so-called \emph{$(1+\varepsilon)$-convex approximation sets} from multi-objective optimization directly corresponds to the notion of parametric $(1+\varepsilon)$-approximation sets~\cite{helfrich.herzel.ea2022ApproximationAlgorithmGeneral}.
Algorithms for computing convex approximation sets are developed in~\cite{papadimitriou.yannakakis2000ApproximabilityTradeoffsOptimal,diakonikolas.yannakakis2008SuccinctApproximateConvex,diakonikolas2011ApproximationMultiobjective,daskalakis.diakonikola2016HowGood,helfrich.ruzika.ea2025EfficientlyConstructing,nemesch.ruzika.ea2025PolynomialTimeInnerApproximation}, an overview is provided in~\cite{herzel.ruzika.ea2021ApproximationMethodsMultiobjective}.

\Citeauthor{helfrich.ruzika.ea2025EfficientlyConstructing}~\cite{helfrich.ruzika.ea2025EfficientlyConstructing} conduct a computational study to compare the performance of algorithms from~\cite{papadimitriou.yannakakis2000ApproximabilityTradeoffsOptimal,helfrich.herzel.ea2022ApproximationAlgorithmGeneral,helfrich.ruzika.ea2025EfficientlyConstructing} on instances of the knapsack problem and the metric travelling salesman problem.
It is demonstrated that adaptive approaches that do not rely on a fixed grid significantly outperform grid-based approaches such as those described in~\cite{papadimitriou.yannakakis2000ApproximabilityTradeoffsOptimal,helfrich.ruzika.ea2025EfficientlyConstructing}.

\medskip

This observation motivates the development of more general algorithms that work in an adaptive fashion.
In this paper, we adapt an exact approach to the approximation setting.
Numerous variants of this exact approach have been developed over the years, both in parametric optimization and in multi-objective optimization, where they are best known as \emph{Benson-type algorithms}~\cite{hamel.lohne.ea2014BensonTypeAlgorithms}.
The basic idea is to characterize the entire set~$X$ of solutions via a convex polyhedron, denoted by~$\lowi$.
Every feasible solution~$x\in X$ maps to a valid halfspace for~$\lowi$ (i.e.\ a halfspace that contains $\lowi$),
and an optimal solution set can be identified by determining all solutions that map to facets of~$\lowi$.
We provide more details on this characterization in Section~\ref{sec::prelims::lowerImage}.
Algorithms that make use of this characterization work by maintaining a polyhedron $\appri\supseteq \lowi$, a so-called \emph{outer approximation polyhedron}~\cite{csirmaz2021InnerApproximationAlgorithm}.
The algorithms work by iteratively refining $\appri$ through repeated intersections with halfspaces that correspond to solutions.

We give a short overview of important (exact) algorithms that work in this fashion.
Due to the aforementioned relation between parametric optimization and multi-objective optimization, we mention literature from both fields.
For a more comprehensive overview, we refer to~\cite{nemesch.ruzika.ea2025SurveyOfExact}.
The first such algorithm is presented by \citeauthor{fernandez-baca.srinivasan1991ConstructingMinimizationDiagrama}~\cite{fernandez-baca.srinivasan1991ConstructingMinimizationDiagrama} for $2$-parametric optimization problems.
A general parametric algorithm is later sketched by \citeauthor{fernandez-baca.seppalainen.ea2004ParametricMultipleSequence} in~\cite{fernandez-baca.seppalainen.ea2004ParametricMultipleSequence}.
At this point in time, it was already considered ``part of the folklore''~\cite{fernandez-baca.seppalainen.ea2004ParametricMultipleSequence}, but, to the best of our knowledge, no formal description or analysis of this algorithm appears anywhere in the literature.

In multi-objective optimization, Benson-type algorithms~\cite{hamel.lohne.ea2014BensonTypeAlgorithms} (sometimes also called ``outer approximation'' or ``inner approximation'' algorithms~\cite{csirmaz2021InnerApproximationAlgorithm}) follow the basic strategy outlined above.
First developed for multi-objective linear programming~\cite{ehrgott.lohne.ea2012DualVariantBenson} and for so-called \emph{vector optimization problems}~\cite{hamel.lohne.ea2014BensonTypeAlgorithms,lohne.rudloff.ea2014PrimalDualApproximation}, they have been generalized to multi-objective integer programming~\cite{borndorfer.schenker.ea2016PolySCIP}, multi-objective combinatorial optimization~\cite{bokler.mutzel2015OutputSensitiveAlgorithmsEnumerating}, and multi-objective mixed-integer programming~\cite{bokler.nemesch.ea2023PaMILOSolverMultiobjective}.
Implementations and numerical experiments in~\cite{lohne.weissing2017VectorLinearProgram,borndorfer.schenker.ea2016PolySCIP,bokler.nemesch.ea2023PaMILOSolverMultiobjective,csirmaz2021InnerApproximationAlgorithm} demonstrate the practical viability of these algorithms.
Sometimes, an additive notion of approximation is used to speed up the practical running time of Benson-type algorithms (for example, in~\cite{hamel.lohne.ea2014BensonTypeAlgorithms}) or as part of an adaption for convex optimization problems~\cite{ehrgott.shao.ea2011ApproximationAlgorithmConvex,liu.ehrgott2018PrimalDualAlgorithms}.
However, in contrast to the multiplicative notion we consider in this paper, it is impossible to provide a polynomial running time bound when using the additive notion.

\section{Our Contribution}

Our main contribution is the adaption of the general strategy of Benson-type algorithms to parametric optimization and approximation.
We present an algorithm that can be used for both exact solving, i.e., finding an optimal solution set, or as an approximation algorithm.
In the exact setting, it runs in output-polynomial time if the underlying non-parametric problem can be solved in polynomial time.
In the approximation setting, the algorithm runs in polynomial time and can achieve an approximation factor arbitrarily close to the best known approximation factor of this underlying problem.

Our algorithm, which is described in Section~\ref{sec::algo}, requires that the parameter set~$\Lambda$ is bounded.
Since we allow arbitrary polyhedra as parameter sets, we describe in Section~\ref{sec::homogenization} how every parametric optimization problem can be transformed (in polynomial time) into an equivalent parametric optimization problem where the parameter set is bounded.
In general, our algorithm can be used to approximate a much bigger class of parametric optimization problems than all previously developed parametric approximation algorithms.
In particular, our algorithm is the first parametric approximation algorithm that can deal with negative parameter dependencies in the parametric value function~$F_\lambda$ and with general polyhedra as parameter sets.
We briefly discuss the different problem assumptions between our algorithm and previously developed algorithms in Section~\ref{sec::prelims::assumptions}.

Furthermore, in Section~\ref{sec::maxNotApx}, we consider a particular class of parametric maximization problems that has not been considered in the literature before.
For this class, our assumptions are slightly relaxed and we show that, under these relaxed assumptions, some parametric maximization problems can never be $(1-\varepsilon)$-approximated in polynomial time, for any constant value of~$\varepsilon$.
This narrows down which assumptions are sufficient to ensure that a parametric optimization problem can be approximated by a constant factor, Table~\ref{sota::table::results} in Section~\ref{sec::prelims::assumptions} provides an overview.

\section{Preliminaries}\label{sec::prelims}

For a mathematical object~$x$, for example a problem instance or a matrix, $\langle x\rangle$ denotes the binary encoding length of~$x$.
We use the well-known operators~$\Omega$ and~$\bigO$ from Landau-notation to denote asymptotic lower and upper bounds, respectively.
By $f(x)\in \bigO(\poly(x))$ we denote that there exists a polynomial function of~$x$ that bounds the function~$f$ asymptotically from above, but where we are not interested in its degree.
For two sets $S_1,S_2\subseteq \mathbb{R}^n$, we denote their Minkowski-sum by $S_1 + S_2 \coloneq \{s_1+s_2 : s_1\in S_1,s_2\in S_2\}$.
By~$\mathbb{R}_{\geq0}$ we denote all non-negative real numbers, and by~$\mathbb{R}_{>0}$ and~$\mathbb{N}_{>0}$ we denote all strictly positive real numbers and integers, respectively.
Analogously denoted are the non-positive and strictly negative real numbers by~$\mathbb{R}_{\leq0}$ and~$\mathbb{R}_{<0}$.

\medskip

We outline some basic concepts for (rational) polyhedra that are relevant in this paper.
For a more comprehensive introduction, we refer to~\cite{grotschel.lovasz.ea1988RationalPolyhedra,ziegler1995LecturesPolytopes,kaibel2011BasicPolyhedralTheory}.

For $n\in\mathbb{N}_{>0}$, $a\in\mathbb{Q}^{n}\setminus\{0\}$, and $b\in\mathbb{Q}$, the set $\{x\in\mathbb{R}^{n}:a^\T x = b\}$ is called a \emph{hyperplane}, and the set $\{x\in\mathbb{R}^{n}:a^\T x \geq b\}$ is called a \emph{halfspace}.
A \emph{(convex) polyhedron} is an intersection of finitely many halfspaces.
The dimension of a polyhedron is the dimension of its affine hull as a subspace.
A bounded polyhedron is called a \emph{polytope}.

A \emph{face} of an $n$-dimensional polyhedron $\mathcal{P}$ is the non-empty intersection of~$\mathcal{P}$ and a hyperplane $\{x\in\mathbb{R}^{n}:a^\T x = b\}$ for which $\mathcal{P}\subseteq\{x\in\mathbb{R}^{n}:a^\T x \geq b\}$.
Every face is also a convex polyhedron.
A face of dimension zero is called a \emph{vertex} of $\mathcal{P}$, and a face of dimension $n-1$ is called a \emph{facet} of $\mathcal{P}$.

A polyhedron $\mathcal{P}$ that is represented by an intersection of halfspaces is said to be in \emph{\Hpoly}.
A polyhedron $\mathcal{P}$ can also be in \emph{\Vpoly}:
It is then represented by two finite sets $S_1,S_2\subseteq\mathbb{Q}^n$ with $\mathcal{P}=\conv (S_1)+\cone (S_2)$, where the operators $\conv$ and $\cone$ denote the convex and conic hull, respectively.
The elements of~$S_1$ are the vertices of $\mathcal{P}$, the elements of $S_2$ are the so-called \emph{(extreme) rays} of~$\mathcal{P}$, and together these elements are the \emph{generators} of~$\mathcal{P}$.

Transferring an \Hpoly\ into a \Vpoly\ is called a \emph{vertex enumeration}, and a transfer in the opposite direction is called a \emph{facet enumeration}.
For fixed~$n$, both transfer operations can be done in polynomial time~\cite{chazelle1993OptimalConvexHull}, and the encoding length of the resulting representation is bounded by a polynomial in the encoding length of the original representation~\cite{grotschel.lovasz.ea1988RationalPolyhedra}.
For an \Hpoly, the operation of removing halfspaces that are not needed for a complete description of the polyhedron is called \emph{redundancy removal}.

\medskip

We conclude this section with a short introduction to multi-objective optimization and its relevance for parametric optimization.
An in-depth introduction to multi-objective optimization can be found in~\cite{ehrgott2005multicriteriaOptimization}.
A \emph{multi-objective optimization problem} is a problem of the form $\min_{x\in X}f(x)$ (or $\max_{x\in X} f(x)$ in the maximization case), where~$X\subseteq\mathbb{R}^n$ is a set of feasible solutions and $f: X \to \mathbb{R}^k$ for some $k\in\mathbb{N}$ with $k\geq2$ is a (vector-valued) objective function.

A solution~$x\in X$ is called \emph{efficient} and is image~$f(x)$ is called \emph{non-dominated} if there is no other solution $x'\in X$ such that $f(x')\neq f(x)$ and $f(x')\leq f(x)$ (or $f(x')\geq f(x)$ in the maximization case).
The goal when solving a multi-objective optimization problem is to find the set of non-dominated images~$Y_{N}$ and, for each image~$y$ in~$Y_N$, a solution~$x\in X$ with $f(x)=y$.
A special subset of the non-dominated images are the \emph{non-dominated extreme points}, i.e., the vertices of the \emph{Edgeworth-Pareto hull} $\conv{(f(X))}+\mathbb{R}^k_{\geq0}$ (or $\conv{(f(X))}+\mathbb{R}^k_{\leq0}$ in the maximization case).

It is well-know that finding the non-dominated extreme points is equivalent to solving a corresponding $(k-1)$-parametric optimization problem $(\mathbb{R}^{k-1}_{\geq0}, X, f)$~\cite{isermann1977EnumerationSetAll,przybylski.gandibleux.ea2009RecursiveAlgorithmFinding}.
Often, this problem is alternatively defined by the parameter set
\[\left\{\lambda\in\mathbb{R}^{k-1}_{\geq0}: \sum_{i=1}^{k+1}\lambda_i=1\right\},
\]
solution set $X$, and objective function
\[
    x\in X \mapsto {\left(f_1(x)-f_k(x),\ldots, f_{k-1}(x)-f_k(x),f_k(x)\right)}^\T.
\]
\subsection{Assumptions}\label{sec::prelims::assumptions}

We make several assumptions to ensure that an instance~$\Pi$ of a parametric optimization problem can be approximated (or solved exactly).
Many of these or similar assumptions are commonly made in parametric approximation~\cite{helfrich.herzel.ea2022ApproximationAlgorithmGeneral}.

First, we assume that~$f$ is computable in polynomial time and that $\Pi(\lambda)$ is bounded for every $\lambda\in\Lambda$.
Furthermore, for each parametric optimization problem, we assume that there exists a polynomial~$\poly$ of bounded degree such that, for every instance~$\Pi$ and every $\lambda\in\Lambda$ of this instance, there is at least one optimal solution~$x\in X$ of~$\Pi(\lambda)$ with encoding length $\langle x \rangle \in \bigO(\poly\langle \Pi\rangle)$.
The assumption of the existence of an optimal solution with polynomial encoding length is satisfied for many problems, e.g., for every parametric combinatorial optimization problem or parametric mixed-integer linear program.
For a given instance $\Pi$, we denote by~$\orsols\subseteq X$ the set of all solutions (not necessarily optimal) with encoding length polynomial in $\langle \Pi \rangle$.

We also assume that the parameter set $\Lambda$ is $p$-dimensional, and that there are no redundant inequalities in the system $Q\lambda\geq d$.
Although this assumption is not strictly necessary, it notably streamlines several parts of this paper; and it can always be ensured by redundancy removal techniques.

Another important assumption is that we consider the number~$p$ of parameters to be a constant.
If~$p$ is a variable part of the input, there are problems where the cardinality of every approximation set grows exponentially in the size of a problem instance.
Although this fact seems to be folklore, to the best of our knowledge, no example has been provided in the literature.
Therefore, we provide Example~\ref{example::pMustBeConstant} in the appendix.

If we are interested in obtaining an approximation set, it needs to be ensured that, for every $\lambda\in\Lambda$, the non-parametric problem $\Pi(\lambda)$ can be approximated.
Therefore, we assume that $F_\lambda(x)\geq 0$ for every $x\in X$ and $\lambda\in\Lambda$.
Note that for maximization problems, the weaker assumption $F^*(\lambda)\geq0$ may also be enough to ensure that~$\Pi(\lambda)$ can be approximated.
However, we show in Section~\ref{sec::maxNotApx} that this weaker assumption is, in general, not sufficient to ensure the existence of $(1-\varepsilon)$-approximation sets of polynomial cardinality.

The last assumption is that an oracle for~$\Pi$ is available.
An \emph{$\alpha$-approximation oracle~$\oracle_{\alpha}$} for an instance~$\Pi$ of a parametric minimization problem~$\Pi$ and  $\alpha\geq1$ is a black-box algorithm that takes as input a $\lambda\in\Lambda$ and returns a solution $x\in\orsols$ such that $F_\lambda(x)\leq\alpha \cdot F^*(\lambda)$.
We call it \emph{exact} if $\alpha=1$.
In the maximization case, $0<\alpha\leq1$ and, on input~$\lambda\in\Lambda$, an oracle~$\oracle_{\alpha}$ must return a solution~$x\in\orsols$ such that $F_\lambda(x)\geq\alpha \cdot F^*(\lambda)$.

\noindent
Assumption~\ref{assumption::general} collects all our assumptions for later reference:
\begin{assumption}\label{assumption::general}
    For every instance~$\Pi$ of a parametric optimization problem, we assume the following:
    \begin{enumerate}[label=(\alph*)]
        \item The function $f$ is polynomial-time computable and, for every $\lambda\in\Lambda$, there exists an optimal solution~$x\in\orsols$ for~$\Pi(\lambda)$.
        \item The number~$p$ of parameters is a constant. Furthermore, $\Lambda$ is $p$-dimensional and the system~$Qx\geq d$ contains no redundant inequalities.
        \item If we are interested in a $(1+\varepsilon)$-approximation set for $\varepsilon>0$, $F_\lambda(x)\geq 0$ holds for all~$\lambda\in\Lambda$ and all~$x\in X$.
        \item An exact or $\alpha$-approximation oracle~$\oracle_\alpha$ is available.
    \end{enumerate}
\end{assumption}

\medskip

All existing work in the parametric approximation literature makes stricter assumptions than we do in Assumption~\ref{assumption::general}.
Most notable, they share the assumption that, for every $x\in X$, not only $F_\lambda(x)\geq0$ must hold, but also the stronger assumption that $f_i(x)\geq0$ must hold for each $i\in\{1,\ldots,p+1\}$  (see~\cite{diakonikolas2011ApproximationMultiobjective,daskalakis.diakonikola2016HowGood,bazgan.herzel.ea2022ApproximationAlgorithmGeneral,helfrich.herzel.ea2022ApproximationAlgorithmGeneral,helfrich.ruzika.ea2025EfficientlyConstructing}).
Furthermore, they have the rather strict requirements for the parameter set that $\Lambda=\mathbb{R}^p_{\geq0}$ must hold (see~\cite{diakonikolas2011ApproximationMultiobjective,daskalakis.diakonikola2016HowGood,bazgan.herzel.ea2022ApproximationAlgorithmGeneral,helfrich.herzel.ea2022ApproximationAlgorithmGeneral,helfrich.ruzika.ea2025EfficientlyConstructing}).\footnote{%
In some papers, these assumptions are formulated slightly different, but are always equivalent to our formulation.}

In contrast, Assumption~\ref{assumption::general} allows to consider a far larger class of problems.
With Algorithm~\ref{algo::apx} from Section~\ref{sec::algo} and an oracle~$\oracle_\alpha$, we can always find a $(1+\varepsilon)\alpha$-approximation set of polynomial cardinality for this larger class of problems.
Note that an exact, but not necessarily polynomial-time oracle~$\oracle_1$ can always be constructed using a brute-force approach.
Therefore, for every $\varepsilon>0$, our Algorithm implies the existence of an $(1+\varepsilon)$-approximation set of polynomial cardinality for every problem satisfying Assumption~\ref{assumption::general}, even if no polynomial-time oracle exists.
Table~\ref{sota::table::results} summarizes the old and new problem classes for which this result holds.

\begin{table}
\centering
\begin{threeparttable}
\caption{%
    Currently known results on whether an assumption ensures that an approximation set of polynomial cardinality exists.
    For each entry, we only cite the most general result found in literature.
    Bold entries are new results that are shown in this paper.
    The last two columns are only relevant for maximization problems.%
\label{sota::table::results}}
\begin{tabular}{l c l l l l l l}\toprule
 \multirow{2}{*}{Assumption} & $F$: & \multicolumn{2}{c}{$f_{1,\ldots,p+1}(x)\geq0$} & \multicolumn{2}{c}{$F_{\lambda}(x)\geq0$} & \multicolumn{2}{c}{$F^*(\lambda)\geq0$} \\
\cmidrule(lr){3-4}\cmidrule(lr){5-6}\cmidrule(lr){7-8}
& $\Lambda$: & $\mathbb{R}^p_{\geq0}$ & polyh. & $\mathbb{R}^p_{\geq0}$ & polyh. & $\mathbb{R}^p_{\geq0}$ & polyh.\\
\midrule
\multirow{2}{*}{minimization} & $p=1$ & yes~\cite{bazgan.herzel.ea2022ApproximationAlgorithmGeneral} & $\text{yes}^\dagger$\,\cite{bazgan.herzel.ea2022ApproximationAlgorithmGeneral} & \textbf{yes} & \textbf{yes} & - & - \\
& $p>1$ & yes~\cite{helfrich.herzel.ea2022ApproximationAlgorithmGeneral} & \textbf{yes} & \textbf{yes} & \textbf{yes} & - & - \\[5pt]
\multirow{2}{*}{maximization} & $p=1$ & yes~\cite{bazgan.herzel.ea2022ApproximationAlgorithmGeneral} & $\text{yes}^\dagger$\,\cite{bazgan.herzel.ea2022ApproximationAlgorithmGeneral} & \textbf{yes} & \textbf{yes} & open & open \\
& $p>1$ & yes~\cite{helfrich.herzel.ea2022ApproximationAlgorithmGeneral} & \textbf{yes} & \textbf{yes} & \textbf{yes} & \textbf{no} & \textbf{no}\\
\bottomrule
\end{tabular}
\begin{tablenotes}\footnotesize
\item[$\dagger$] The generalization to arbitrary parameter intervals is not developed in~\cite{bazgan.herzel.ea2022ApproximationAlgorithmGeneral}, but we consider it to be straight-forward.
\end{tablenotes}
\end{threeparttable}
\end{table}

\subsection{Polyhedral Characterization}\label{sec::prelims::lowerImage}

We now characterize optimal solution sets and approximation sets through polyhedra.
In the exact case, this is not a new concept and has been observed in parametric optimization~\cite{fernandez-baca.srinivasan1991ConstructingMinimizationDiagrama} as well as in multi-objective optimization~\cite{heyde.lohne2008GeometricDualityMultiple}.
We generalize this to the approximation case.

For a solution $x\in X$, we define the hyperplane
\[
    h(x)\coloneqq\left\{(\lambda,z)\in\mathbb{R}^{p+1}:\sum_{i=1}^p \lambda_i f_i(x) - z = - f_{p+1}(x)\right\}
\]
and the halfspace
\[
    H(x)=\left\{(\lambda,z)\in\mathbb{R}^{p+1}:\sum_{i=1}^p \lambda_i f_i(x) - z \geq - f_{p+1}(x)\right\}.
\]

For every~$\lambda\in\Lambda$ and $(\lambda,z)\in h(x)$, the definition of~$h(x)$ implies that $z=F_\lambda(x)$ holds.
Therefore, $z$ can be seen as an upper bound for~$F^*(\lambda)$, and the point $(\lambda,F^*(\lambda))$ must be in~$H(x)$.
Note that, in the maximization case, the inequality in the definition of the halfspace~$H(x)$ is reversed.

\begin{definition}
    For a set~$S\subseteq X$, we define the set
    \[
        \appri(S)\coloneqq \bigcap_{x\in S}H(x) \cap \left\{(\lambda,z)\in \Lambda \times \mathbb{R}\right\}.
    \]
    The \emph{lower image} of a parametric problem with feasible set~$X$ is the set $\lowi \coloneqq \appri(X)$.
\end{definition}
The term \emph{lower image} originally stems from multi-objective linear programming~\cite{lohne.rudloff.ea2014PrimalDualApproximation}, but we adopt it for parametric optimization.
For a given instance~$\Pi$, the set~$\orsols$ always is an optimal solution set (see Assumption~\ref{assumption::general}) and, therefore, $\lowi=\appri(\orsols)$ holds.
In addition,~$\orsols$ is finite by definition because the encoding lengths of its solutions is finite.
Thus, $\lowi$ can be described by a finite set of halfspaces and is a polyhedron.

\medskip

Let $S\subseteq X$ be an arbitrary subset of~$X$.
Clearly, $\lowi\subseteq \appri(S)$.
A face $\varphi$ of $\appri(S)$ is called \emph{maximal} if there are no two points $(\lambda,z),(\lambda,z')\in\varphi$ such that $z<z'$.
For each point~$(\lambda,z)$ in a maximal face of~$\appri(S)$, there is at least one solution~$x\in S$ such that $z=\min_{x\in S}F_\lambda(x)$.
Therefore, if $\appri(S)=\lowi$, then $S$ must be an optimal solution set.
Moreover, if the set $h(x)\cap \lowi$ is a unique facet of $\lowi$ for each~$x\in X$, then~$\appri(S)$ describes~$\lowi$ without redundancies.
In this case, $S$ is an optimal solution set of minimum cardinality.

\medskip

In the case where we are interested in a $(1+\varepsilon)$-approximation set, we know that $F^*(\lambda)\geq0$ holds for every $\lambda\in\Lambda$ by Assumption~\ref{assumption::general}.
Let
\[
    \lowieps\coloneqq\{(\lambda,(1+\varepsilon)z):(\lambda,z)\in\lowi\}.
\]
Then, $\lowi\subseteq\lowieps$ holds by construction.
If (and only if) $\appri(S)\subseteq\lowieps$, the set~$S$ is a $(1+\varepsilon)$-approximation set:
for every~$\lambda\in\Lambda$, the point~$(\lambda,z)$ in a maximal facet of~$\appri(S)$ is also in $\lowieps$ and, therefore, $z\leq(1+\varepsilon)F^*(\lambda)$.
This observation is used in the design of our algorithm in Section~\ref{sec::algo}.

\section{The General Parametric Approximation Algorithm}\label{sec::algo}

In this section, we describe our algorithm under Assumption~\ref{assumption::general} and the additional assumption that $\Lambda$ is a polytope.
As we show in Section~\ref{sec::homogenization}, every instance of a parametric optimization problem can be transformed (in polynomial time) into another instance of the same parametric optimization problem with identical optimal solution sets and approximation sets, and a polytope as parameter set.
Therefore, if this transformation is applied first, our algorithm can be used for every instance of a parametric optimization problem, as long as it satisfies Assumption~\ref{assumption::general}.
We first give an informal description of the algorithm before formally stating and analyzing it.

The goal of the algorithm is to compute either an approximation set or an optimal solution set.
For this, it only needs access to an oracle for $\Pi(\lambda)$.
In the case that an approximation set is the goal, an approximation oracle $\oracle_\alpha$ with $\alpha\geq1$ is sufficient.
For every input $\varepsilon>0$ and problem instance~$\Pi$, the algorithm computes a $(1+\varepsilon)\alpha$-approximation set while calling $\oracle_\alpha$ only a polynomial number of times.
If an optimal solution set is the goal, it can be seen as the special case where $\varepsilon=0$ and the oracle is exact, so $\alpha=1$.
The number of calls to $\oracle_1$ as well as the remaining running time in this case is output-polynomial, i.e., bounded by a polynomial in the size of the input and the cardinality of the output set (see~\cite{johnson.yannakakis.ea1988GeneratingAllMaximal}).%
\footnote{We ignore the remaining case where $\varepsilon=0$ and $\alpha>1$, as this results in an $\alpha$-approximation set without ensuring polynomial running time.}

Our algorithm works iteratively.
It maintains a set $S\subseteq X$ of solutions, beginning with one arbitrary solution $x\in \orsols$.
At the beginning of each iteration, the algorithm checks if the current set $S$ is a $(1+\varepsilon)\alpha$-approximation set.
Recall the polyhedra~$\appri(S)$ and~$\lowieps$ from Section~\ref{sec::prelims::lowerImage}.
If $\appri(S)\subseteq\lowiepsarg{(1+\varepsilon)\alpha}$, then $S$ is a $(1+\varepsilon)\alpha$-approximation set.
Therefore, our algorithm needs to determine whether $\appri(S)\subseteq\lowiepsarg{(1+\varepsilon)a}$ holds.

For each vertex $(\lambda,z)$ of $\appri(S)$, the algorithm calls the oracle~$\oracle_\alpha(\lambda)$, which returns a solution~$x$.
If $(1+\varepsilon)\cdot F_\lambda(x)< z$, then $x$ is not yet $(1+\varepsilon)$-approximated in~$S$.
Since~$x$ itself is only an $\alpha$-approximation for~$\Pi(\lambda)$, it is not ensured that the optimal solution value of~$\Pi(\lambda)$ is $(1+\varepsilon)\alpha$-approximated in~$S$.
The algorithm then adds~$x$ to~$S$ to ensure that~$\Pi(\lambda)$ is $(1+\varepsilon)\alpha$-approximated and begins with the next iteration.
For the polyhedron~$\appri(S)$, adding~$x$ to~$S$ means that~$\appri(S)$ is intersected with the halfspace~$H(x)$, which improves the upper bound on~$F^*$ not only for~$\lambda$ itself, but also in a region around it (potentially up to the entire parameter set~$\Lambda$).

If, however, $(1+\varepsilon)\cdot F_\lambda(x)\geq z$ holds for every solution~$x$ returned by $\oracle_\alpha$ in an iteration, convexity ensures that $\appri(S)\subseteq\lowiepsarg{(1+\varepsilon)\alpha}$.
This verifies that~$S$ is a $(1+\varepsilon)\alpha$-approximation set and the algorithm terminates.
To avoid redundant calls to the oracle, the algorithm keeps track of all values of~$\lambda$ for which~$\oracle_\alpha(\lambda)$ has been called in previous iterations and skips those.

After verifying that~$S$ is a $(1+\varepsilon)\alpha$-approximation set, $S$ might contain some solutions for which the hyperplane~$h(x)$ does not contain a facet of~$\appri(S)$.
These can be removed in an optional cleaning step.
In the exact case (i.e., $\varepsilon=0$ and $\alpha=1$), this ensures that~$S$ has minimum cardinality.

Algorithm~\ref{algo::apx} formally describes our algorithm, and Figure~\ref{figure::algo::iteration} illustrates an example for an iteration.
Note that the oracle~$\oracle_\alpha$ is considered to be a fixed part of the algorithm and not part of the input.
For the rest of this section, we prove several important properties of Algorithm~\ref{algo::apx}:
First, we show in Theorem~\ref{theorem::algo::apxGuarantee} that, after termination of Algorithm~\ref{algo::apx}, the set~$S$ is indeed a $(1+\varepsilon)\alpha$-approximation set.
Then, we analyze the running time properties for the exact case in Section~\ref{sec::algo::exact}, and, afterwards,  for the approximation case in Section~\ref{sec::algo::apx}.

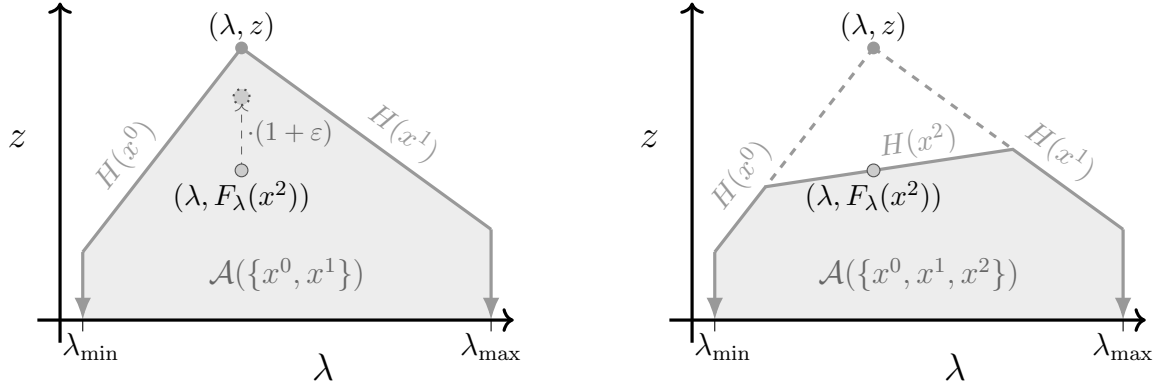
\begin{figure}
    \begin{minipage}[b]{0.475\textwidth}
        \centering
        \begin{tikzpicture}[%
    scale=0.6,
    lowiedge/.style={color=black!40, very thick},
    lowiextended/.style={lowiedge, dashed},
    lowiray/.style={lowiedge, {-Latex}},
    lowifill/.style={draw=none, fill=black!7},
    tick/.style={thin},
]
    \coordinate (v0) at (0.5,0);
    \coordinate (v1) at (0.5,1.5);
    \coordinate (v2) at (4,6);
    \coordinate (v3) at (9.5,2);
    \coordinate (v4) at (9.5,0);

    \coordinate (opt) at (4,3.3);
    \coordinate (eps) at ($(opt)!0.6!(v2)$);

    \draw[lowifill] (v0) to (v1) to (v2) to (v3) to (v4) to cycle;
    \draw[lowiray] (v1) to (v0);
    \draw[lowiray] (v3) to (v4);

    \draw[lowiedge] (v1) to node[pos=0.5,sloped,xshift=-10pt,yshift=10pt] {\small$H(x^0)$} (v2);
    \draw[lowiedge] (v2) to node[pos=0.5,sloped,xshift=10pt,yshift=10pt] {\small$H(x^1)$}  (v3);

    \draw[->,shorten >=3pt, dashed,black!60] (opt) to %
    node[pos=0.5,right,xshift=-2pt] {\footnotesize%
    $\cdot(1+\varepsilon)$}(eps);
    \draw[draw=black!70, fill=black!20] (opt) circle (4pt);
    \draw[draw=black!70, fill=black!20,dotted,thick] (eps) circle (5pt);
    \draw[draw=none, fill=black!40, color=black!40] (v2) circle (4pt);
    \node[yshift=8pt] at (v2) {\small$(\lambda,z)$};
    \node[yshift=-10pt] at (opt) {\small$(\lambda,F_\lambda(x^2))$};

    \draw[very thick,->] (-0.5,0) to node[pos=0.6, below, yshift=-8pt] {\large$\lambda$} (10.0, 0);
    \draw[very thick,->] (0,-0.5)to node[pos=0.6, left, xshift=-8pt] {\large$z$} (0, 7.0);
    \draw[tick] (0.5,0) to (0.5,-0.25);
    \node[yshift=-10pt,xshift=3pt] at (0.5,0) {$\lambda_{\min}$};
    \draw[tick] (9.5,0) to (9.5,-0.25);
    \node[yshift=-10pt] at (9.5,0) {$\lambda_{\max}$};

    \node[color=black!60] at (5,1) {$\appri{\left(\left\{x^0,x^1\right\}\right)}$};


\end{tikzpicture}
    \end{minipage}
    \hfill
    \begin{minipage}[b]{0.475\textwidth}
        \centering
        \begin{tikzpicture}[%
    scale=0.6,
    lowiedge/.style={color=black!40, very thick},
    lowiextended/.style={lowiedge, dashed},
    lowiray/.style={lowiedge, {-Latex}},
    lowifill/.style={draw=none, fill=black!7},
    tick/.style={thin},
]
    \coordinate (v0) at (0.5,0);
    \coordinate (v1) at (0.5,1.5);
    \coordinate (v2) at (4,6);
    \coordinate (v3) at (9.5,2);
    \coordinate (v4) at (9.5,0);

    \coordinate (opt) at (4,3.3);
    \coordinate (eps) at ($(opt)!0.6!(v2)$);

    \coordinate (leftintersec) at ($(v1)!0.32!(v2)$);
    \coordinate (rightintersec) at ($(leftintersec)!2.29!(opt)$);

    \draw[lowifill] (v0) to (v1) to (leftintersec) to (rightintersec) to (v3) to (v4) to cycle;
    \draw[lowiray] (v1) to (v0);
    \draw[lowiray] (v3) to (v4);

    \draw[lowiedge,draw=none] (v1) to node[pos=0.4,sloped,xshift=-10pt,yshift=10pt] {\small$H(x^0)$} (v2);
    \draw[lowiedge,draw=none] (v2) to node[pos=0.6,sloped,xshift=10pt,yshift=10pt] {\small$H(x^1)$}  (v3);
    \draw[lowiedge,draw=none] (leftintersec) to node[pos=0.63,sloped,yshift=8pt] {\small$H(x^2)$}  (rightintersec);
    \draw[lowiedge] (v1) to (leftintersec) to (rightintersec) to (v3);
    \draw[lowiextended] (v2) to (leftintersec);
    \draw[lowiextended] (v2) to (rightintersec);

    \draw[draw=black!70, fill=black!20] (opt) circle (4pt);
    \draw[draw=none, fill=black!40, color=black!40] (v2) circle (4pt);
    \node[yshift=8pt] at (v2) {\small$(\lambda,z)$};
    \node[yshift=-10pt] at (opt) {\small$(\lambda,F_\lambda(x^2))$};

    \draw[very thick,->] (-0.5,0) to node[pos=0.6, below, yshift=-8pt] {\large$\lambda$} (10.0, 0);
    \draw[very thick,->] (0,-0.5)to node[pos=0.6, left, xshift=-8pt] {\large$z$} (0, 7.0);
    \draw[tick] (0.5,0) to (0.5,-0.25);
    \node[yshift=-10pt,xshift=3pt] at (0.5,0) {$\lambda_{\min}$};
    \draw[tick] (9.5,0) to (9.5,-0.25);
    \node[yshift=-10pt] at (9.5,0) {$\lambda_{\max}$};


    \node[color=black!60] at (5,1) {$\appri{\left(\left\{x^0,x^1,x^2\right\}\right)}$};
\end{tikzpicture}
    \end{minipage}
    \caption{%
        A single refining step in Algorithm~\ref{algo::apx} (for a $1$-parametric problem with $\Lambda=\left[{\lambda_{\min},\lambda_{\max}}\right]$):
        The two solutions~$x^0$ and~$x^1$ are already known and the intersection of their halfspaces~$H(x^0)$ and~$H(x^1)$ bounds the polyhedron~$\appri{\left(\left\{x^0,x^1\right\}\right)}$.
        Algorithm~\ref{algo::apx} picks the vertex~$(\lambda, z)$ of~$\appri{\left(\left\{x^0,x^1\right\}\right)}$ and, through $\oracle_{\alpha}$, computes an $\alpha$-approximate solution~$x^2\in\orsols$ for~$\Pi(\lambda)$.
        If $(1+\varepsilon)\cdot F_\lambda(x^2)< z$, the halfspace $H(x^2)$ is used to refine $\appri(S)$ and cuts off the vertex~$(\lambda, z)$.%
        \label{figure::algo::iteration}
    }
\end{figure}

\begin{algorithm}
    \algrenewcommand\algorithmicrequire{\textbf{Input:}}
    \algrenewcommand\algorithmicensure{\textbf{Output:}}
    \caption{The General Parametric Approximation Algorithm}\label{algo::apx}
    \begin{algorithmic}[1]
    \Require{Instance~$\Pi=(\lambda,X,f)$ of a parametric optimization problem and $\varepsilon\geq 0$%
    }
    \Ensure{$S\subseteq X$ such that $\appri(S)\subseteq\lowiepsarg{(1+\varepsilon)a}$}
    \State $\lambda^0\gets$ arbitrary $\lambda\in\Lambda$
    \State $V\gets\{\lambda^0\}$
    \State $x^0\gets\oracle_{\alpha}(\lambda^0)$
    \State $S\gets\{x^0\}$
    \State Compute $\vertx\left(\appri(S)\right)$
    \For{$(\lambda,z)\in \vertx\left(\appri(S)\right)$}
    \If{$\lambda\notin V$}
        \State $V\gets V\cup\{\lambda\}$
        \State $x\gets\oracle_{\alpha}(\lambda)$
        \If{$(1+\varepsilon)F_{\lambda}(x)<z$}
            \State $S\gets S \cup\{x\}$
            \State \textbf{goto} Step~5
        \EndIf
    \EndIf
    \EndFor
    \State \textbf{optional:} Remove redundant solutions from $S$
    \State \Return $S$
    \end{algorithmic}
\end{algorithm}

\begin{theorem}\label{theorem::algo::apxGuarantee}
    For every $\varepsilon\geq0,\alpha\geq1$, Algorithm~\ref{algo::apx} computes a $(1+\varepsilon)\alpha$-approximation set for~$\Pi$.
\end{theorem}
\begin{proof}
    Our proof works by contradiction: we use the convexity of the maximal facets of~$\appri(S)$ to show that, if there is any~$\lambda\in\Lambda$ for which~$S $ does not contain a $(1+\varepsilon)\alpha$-approximate solution of~$\Pi(\lambda)$, the algorithm could not have terminated.

    Let $\lambda\in\Lambda$, and let $(\lambda,z)\in\appri(S)$ be a point in a maximal facet~$\varphi$ of~$\appri(S)$.
    By construction of~$\appri(S)$, there exists a solution~$x\in S$ with $\varphi\subseteq h(x)$, in particular $F_{\lambda}(x)=z$ by definition of~$h(x)$ (see Section~\ref{sec::prelims::lowerImage}).
    Let $v^1=(\lambda^1,z^1),\ldots,v^k=(\lambda^k,z^k)$ be the vertices of $\varphi$.
    Note that $z^i=F_{\lambda^i}(x)$ for each $i\in\{1,\ldots,k\}$.
    Because~$\varphi$ is convex, there is an $\ell\in\mathbb{R}_{\geq0}^k$ with $\|\ell\|_1=1$ such that $\sum_{i=1}^k \ell_i v^i = (\lambda, z)$ and, thus, $F_\lambda(x)=\sum_{i=1}^k \ell_i F_{\lambda^i}(x)$.

    For the sake of contradiction, assume that there is a solution~$x^*\in X$ with $(1+\varepsilon) \alpha\cdot F_{\lambda}(x^*)<F_{\lambda}(x)$.
    Since $F_\lambda$ is linear in $\lambda$, we obtain $F_\lambda(x^*)=\sum_{i=1}^k \ell_i F_{\lambda^i}(x^*)$.
    Therefore, $(1+\varepsilon)\alpha\cdot\sum_{i=1}^k \ell_i F_{\lambda^i}(x^*)<\sum_{i=1}^k \ell_i F_{\lambda^i}(x)$, so there must be at least one $i\in\{1,\ldots,k\}$ with $(1+\varepsilon)\alpha\cdot F_{\lambda^i}(x^*)<F_{\lambda^i}(x)$.
    Let $\tilde{\imath}$ denote such an index.

    Since $v^{\tilde{\imath}}$ is a vertex of $\appri(S)$, Algorithm~\ref{algo::apx} must, at some point, have picked a vertex $(\lambda^{\tilde{\imath}},\tilde{z}^{\tilde{\imath}})$ with $\tilde{z}^{\tilde{\imath}}\geq z^{\tilde{\imath}}$ in Step~6 and called $\oracle_\alpha(\lambda^{\tilde{\imath}})$  in Step~9.
    Then, it was either verified that $(1+\varepsilon)\cdot\alpha\cdot F^*(\lambda_{\tilde{\imath}})\geq \tilde{z}^{\tilde{\imath}}$, or an $\alpha$-approximate solution of~$\Pi(\lambda^{\tilde{\imath}})$ was added to~$S$, cutting of~$(\lambda^{\tilde{\imath}},\tilde{z}^{\tilde{\imath}})$ from $\appri(S)$.
    In both cases, it is then ensured that~$S$ contains a $(1+\varepsilon)\alpha$-approximate solution of~$\Pi(\lambda^{\tilde{\imath}})$.

    Since $(\lambda^{\tilde{\imath}},z^{\tilde{\imath}})$ is in the hyperplane~$h(x)$, we know that $F_{\lambda_{\tilde{\imath}}}(x)\leq F_{\lambda_{\tilde{\imath}}}(x')$ for every~$x'\in S$.
    Therefore, $x$ must be a $(1+\varepsilon)\alpha$-approximate solution of~$\Pi(\lambda_{\tilde{\imath}})$.
    This contradicts $(1+\varepsilon)\cdot\alpha\cdot F_{\lambda_{\tilde{\imath}}}(x^*)<F_{\lambda_{\tilde{\imath}}}(x)$.
\end{proof}

\subsection{The Exact Case}\label{sec::algo::exact}

We analyze the running time of Algorithm~\ref{algo::apx} in the exact case, where $\varepsilon=0$ and $\alpha=1$.
A similar running time analysis for the dual Benson algorithm can be found in~\cite{bokler.mutzel2015OutputSensitiveAlgorithmsEnumerating}, and in more detail in~\cite{bokler2018OutputsensitiveComplexityMultiobjective}.
While similar, our analysis slightly deviates from the ones conducted therein.

In this section, we denote by $\novert$, $\nofaces$, and $\nofacet$ the number of maximal vertices, faces, and facets of~$\lowi$, respectively (see Section~\ref{sec::prelims::lowerImage}).
Let $S^*$ be an optimal solution set of minimum cardinality.
Observe that there must be exactly one solution per maximal facet of $\lowi$, i.e.,  $\left|S^*\right|=\nofacet$.
By the upper bound theorem~\cite{seidel1995UpperBoundTheorem}, $\nofaces \in \bigO{\left({(\nofacet+m)}^{\left\lfloor\frac{p}{2}\right\rfloor}\right)}=\bigO{\left({(\left|S^*\right|+m)}^{\left\lfloor\frac{p}{2}\right\rfloor}\right)}$.

We first observe that, in the exact variant, our algorithm fits the skeleton algorithm for outer approximations from~\cite{csirmaz2021InnerApproximationAlgorithm}.
Solving the fixed-parameter problem~$\Pi(\lambda)$ and refining~$\appri(S)$ with a halfspace~$H(x)$ for~$\lowi$ if~$H(x)$ cuts of $(\lambda,z)$ constitutes a \emph{weak point separating oracle} as described in~\cite{csirmaz2021InnerApproximationAlgorithm}.
This allows us to transfer the following result from~\cite{csirmaz2021InnerApproximationAlgorithm}:
\begin{lemma}[\Citeauthor{csirmaz2021InnerApproximationAlgorithm}~\cite{csirmaz2021InnerApproximationAlgorithm}]\label{lemma::exactfaces}
    For $\varepsilon=0$ and $\alpha=1$, Algorithm~\ref{algo::apx} makes at most $\novert + \nofaces$ calls to the oracle~$\oracle_\alpha$, and adds at most~$\nofaces$ solutions to~$S$.
\end{lemma}

We also observe that, since every solution returned by~$\oracle_\alpha$ is in~$\orsols$, all halfspaces that appear in Algorithm~\ref{algo::apx} have an encoding length in~$\bigO{\left(\poly\langle\Pi\rangle\right)}$.
Therefore, every vertex that is picked in Step~6 has polynomial encoding length~\cite{bokler.mutzel2015OutputSensitiveAlgorithmsEnumerating}.
As a consequence, we can assume that if~$\oracle_\alpha$ is a polynomial-time algorithm, then every call to~$\oracle_\alpha$ in Algorithm~\ref{algo::apx} runs in time~$\bigO{\left(\poly\langle\Pi\rangle\right)}$
(see also~\cite{bokler.mutzel2015OutputSensitiveAlgorithmsEnumerating}).

\begin{theorem}\label{theorem::algo::runningtimeExact}
    For $\varepsilon=0$ and $\alpha=1$, Algorithm~\ref{algo::apx} runs in time
    \[
    \bigO{\left(\nofaces\cdot\left(\nofaces^{\lfloor\frac{p}{2}\rfloor}+\nofaces\log\nofaces+T_{\oracle} \right)\right)},
    \]
    where~$T_{\oracle}$ is a running time bound for the oracle~$\oracle_\alpha$.
\end{theorem}
\begin{proof}
    The proof is split into two parts:
    First, we discuss the accumulated running time from Steps~5 to~12.
    Then, we show that the remaining operations are (asymptotically) dominated by this accumulated running time.

    By Lemma~\ref{lemma::exactfaces}, at most~$\nofaces$ elements are added to~$S$.
    Therefore, Algorithm~\ref{algo::apx} jumps to Step~5 at most $\nofaces$~times.
    We refer to one run through the Steps~5 to~12 as an iteration.

    Let~$S'$ be the set of solutions that are in~$S$ at the start of an arbitrary iteration.
    This iteration starts with Step~5: the enumeration of the vertices of the current polyhedron $\appri(S')$, which is defined by $m+\left|S'\right| \leq m+\nofaces \in \bigO(\nofaces)$ inequalities (recall that, by Assumption~\ref{assumption::general}, there are no redundant inequalities in the description of~$\Lambda$ and, thus, $m\leq\nofaces$).
    The vertex enumeration can then be done in $\bigO{\left(\nofaces^{\lfloor\frac{p}{2}\rfloor}+\nofaces\log\nofaces\right)}$ using the algorithm from~\cite{chazelle1993OptimalConvexHull}.

    \medskip

    In Step~6, the algorithm loops over the vertices of~$\appri(S')$.
    For every vertex~$(\lambda,z)$ of~$\appri(S')$, the algorithm must check if the parameter vector~$\lambda$ has been given to~$\oracle_\alpha$ in a previous iteration, which is done by checking if~$\lambda$ is in~$V$ in Step~7.
    At this point, there are at most $\bigO(\nofaces)$ elements in~$V$.
    The reason is the following:
    Elements are only added to~$V$ in Step~8 (and once in Step~2).
    This is always accompanied by a call to~$\oracle_\alpha$ in Step~9.
    Hence, the bound from Lemma~\ref{lemma::exactfaces} for the number of oracle calls can also be applied to the number of elements added to~$V$, which is, thus, upper-bounded by $\novert+\nofaces\in\bigO(\nofaces)$.
    By maintaining~$V$ as a balanced binary search tree, every single check in Step~7 then needs at most $\bigO(\log\nofaces)$ time.
    The same reasoning also applies to the insertions in Step~8.

    There are two possible outcomes of the check in Step~7.
    Either, $\lambda$ is in~$V$ and the loop from Step~6 proceeds to the next vertex of~$\appri(S')$, or~$\lambda$ is not in~$V$ and the algorithm proceeds to Step~8.
    We consider both cases separately.
    For the first case, observe that all the vertices of~$\appri(S')$ have distinct $\lambda$~values by construction.
    Since there are at most $\bigO(\nofaces)$ elements in~$V$, the first case can then only be encountered $\bigO(\nofaces)$ times in a single iteration.
    Thus, its accumulated running time is at most $\bigO(\nofaces\log\nofaces)$, which vanishes in the running time bound for the vertex enumeration from Step~5.

    This leaves the second case.
    By Lemma~\ref{lemma::exactfaces}, this case is encountered at most $\bigO(\nofaces)$ times over all iterations together.
    Each time, Step~7 and Step~8 run in $\bigO(\log\nofaces)$ time and Step~9 runs in $\bigO(T_{\oracle})$ time.
    The remaining Steps~10 to~12 run in $\bigO(1)$ time.

    To summarize up to this point: Algorithm~\ref{algo::apx} has at most $\bigO(\nofaces)$ iterations.
    In each iteration, the vertex enumeration accumulates running time up to $\bigO{\left(\nofaces^{\lfloor\frac{p}{2}\rfloor}+\nofaces\log\nofaces\right)}$.
    The running time of the first case from above in a single iteration is at most $\bigO(\nofaces\log\nofaces)$.
    Additionally, over all iterations together, the running times of Step~7 and of Steps~8 to~12 are all in $\bigO{\left(\nofaces\cdot(\log\nofaces+T_{\oracle})\right)}$.
    Put together and simplified, this amounts to
    \[
    \bigO{\left(\nofaces\cdot\left(\nofaces^{\lfloor\frac{p}{2}\rfloor}+\nofaces\log\nofaces+T_{\oracle} \right)\right)}.
    \]


    This leaves the running times of Steps~1 to~4 and Step~13.
    Determining an arbitrary~$\lambda^0$ in Step~1 is not harder than finding a vertex of~$\Lambda$ and can be counted as one additional vertex enumeration.
    Step~3 needs $\bigO(T_{\oracle})$ time, and Steps~2 and~4 need $\bigO(1)$ time.
    Both running times are all dominated by the term above.

    The overhead of removing redundant solutions from~$S$ in Step~13 is also dominated.
    This can be seen by looking at a simple brute force approach for the redundancy removal:%
\footnote{Specialized redundancy removal algorithms might achieve better running times.}
    After the last iteration, the set of vertices of~$\lowi$ is known.
    A facet enumeration for~$\lowi$ can then be performed in $\bigO{\left(\novert^{\lfloor\frac{p}{2}\rfloor}+\novert\log\novert\right)}\subseteq \bigO{\left(\nofaces^{\lfloor\frac{p}{2}\rfloor}+\nofaces\log\nofaces\right)}$ and results in at  most $\bigO{\left(\novert^{\lfloor\frac{p}{2}\rfloor}\right)}\subseteq \bigO{\left(\nofaces^{\lfloor\frac{p}{2}\rfloor}\right)}$ facets~\cite{seidel1995UpperBoundTheorem}.
    For each~$x\in S$, we check whether~$h(x)$ contains a facet.
    If not, $x$ is removed from~$S$.
    This requires at most $\bigO{\left(\left|S\right|\nofaces^{\lfloor\frac{p}{2}\rfloor}\right)}\subseteq \bigO{\left(\nofaces\cdot\nofaces^{\lfloor\frac{p}{2}\rfloor}\right)}$ pair-wise comparisons between hyperplanes and facets.
    By assuming that $S$ is maintained as an array or a linked list, every single removal operation can be done in $\bigO(1)$.
    The entire running time of the brute-force redundancy removal then is in $\bigO{\left(\nofaces \cdot \left(\nofaces^{\lfloor\frac{p}{2}\rfloor}+ \log\nofaces\right)\right)}$, which is dominated by the term above.
\end{proof}

\subsection{The Approximation Case}\label{sec::algo::apx}

For the approximation case, Lemma~\ref{lemma::exactfaces} no longer applies.
The arguments regarding the encoding length of the vertices and facets of~$\appri(S)$, however, still apply.
Hence, if~$\oracle_\alpha$ is a polynomial-time algorithm, each call to the oracle in Algorithm~\ref{algo::apx} runs in polynomial time.
Throughout this section, we denote the set~$S$ immediately before the optional redundancy removal in Step~13 by~$\Smax$, and we use~$\smax$ to denote its cardinality.
The cardinality of the set~$S$ returned by Algorithm~\ref{algo::apx} is, therefore, at most~$\smax$.

\begin{theorem}\label{theorem::time::apxout}
    For an $\varepsilon>0$ and $\alpha\geq1$, Algorithm~\ref{algo::apx} is finite and runs in time
    \[
        \bigO{\left(\smax\cdot{\left(\smax+m\right)}^{\left\lfloor\frac{p}{2}\right\rfloor}\cdot\left(p\cdot\log{\left(\smax+m\right)}+ T_{\oracle}\right)\right)},
    \]
    where $T_{\oracle}$ is a running time bound for the oracle $\oracle_\alpha$.
\end{theorem}
\begin{proof}
    The finiteness of the set~$\orsols$ directly proves the finiteness of Algorithm~\ref{algo::apx}:
    Every solution in~$\orsols$ can be added to~$S$ at most once, and, therefore, Algorithm~\ref{algo::apx} jumps to Step~5 at most $\left|\orsols\right|$ times.
    We now analyze the time Algorithm~\ref{algo::apx} spends between one such jump and the next jump.
    Again, we call this an iteration.

    After an arbitrary jump to Step~5, let~$S'$ be the set of currently known solutions.
    By construction, $\left|S'\right|\leq \smax$.
    First, the vertices of the polyhedron~$\appri(S')$ are enumerated in Step~5.
    Since there are at most $\left|S'\right|+m\leq\smax+m$ inequalities defining~$\appri(S')$, this takes time
    $\bigO{\left({(\smax+m)}^{\left\lfloor\frac{p}{2}\right\rfloor} + (\smax+m)\cdot \log (\smax+m)\right)}$~\cite{chazelle1993OptimalConvexHull} and results in at most $\bigO{\left({(\smax+m)}^{\left\lfloor\frac{p}{2}\right\rfloor}\right)}$ vertices~\cite{seidel1995UpperBoundTheorem}.

    For every vertex~$(\lambda,z)$ of~$\appri(S')$, it is checked in Step~7 if~$\lambda$ is already in~$V$.
    Let~$V$ be maintained as balanced binary search tree.
    At this point in Algorithm~\ref{algo::apx}, there are not more than~$\smax$ previous iterations, and in each iteration at most $\bigO{\left({(\smax+m)}^{\left\lfloor\frac{p}{2}\right\rfloor}\right)}$ vertices were enumerated, checked, and subsequently added to~$V$.
    Hence, at most $\bigO{\left({\smax\cdot(\smax+m)}^{\left\lfloor\frac{p}{2}\right\rfloor}\right)}$ elements are in~$V$, and checking for a single~$\lambda$ whether it is in~$V$ (as well as inserting~$\lambda$ into~$V$ in Step~8 if needed) can be done in 
    \[
    \bigO{\left(\log\left({\smax\cdot(\smax+m)}^{\left\lfloor\frac{p}{2}\right\rfloor}\right)\right)}\in \bigO{\left(p\cdot\log\left(\smax+m\right)\right)}.
    \]
    Therefore, the running times of all operations involving~$V$ (checks and insertions in Steps~7 and~8) during a single iteration accumulates to at most
    \[
    \bigO{\left({(\smax+m)}^{\left\lfloor\frac{p}{2}\right\rfloor}\cdot\left(p\log\left(\smax+m\right)\right)\right)},
    \]
    which dominates the running time bound for the vertex enumeration.

    The oracle~$\oracle_\alpha$ is called at most once per vertex of~$\appri(S')$.
    Combined, we get a running time of at most
    \[
    \bigO{\left({(\smax+m)}^{\left\lfloor\frac{p}{2}\right\rfloor} \cdot \left(p\log\left(\smax+m\right) + T_\oracle\right)\right)}
    \]
    for Steps~5 to~12 in a single iteration.
    Multiplying by the number~$\smax$ of iterations yields the desired bound for the running time.
    For the remaining Steps~1 to~4 and Step~13, the same arguments as in the proof of  Theorem~\ref{theorem::algo::runningtimeExact} can be applied to show that their running time bound is dominated by the running time bound above.
    We omit repeating these arguments here.
\end{proof}

The running time of Algorithm~\ref{algo::apx} from Theorem~\ref{theorem::time::apxout} depends on~$\smax$, the maximum cardinality of~$S$ during Algorithm~\ref{algo::apx}.
The following theorem provides an upper bound for~$\smax$.

\begin{theorem}\label{theorem::time::apxpoly}
    In Algorithm~\ref{algo::apx}, $\smax$ is bounded by
    \[
        \smax\in \bigO{\left(m^{\lfloor \frac{p}{2}\rfloor\cdot\lceil \frac{p}{2}\rceil}\cdot\poly{(\langle \Pi \rangle,\varepsilon^{-1})}^{p+1} \right)}.
    \]
\end{theorem}

Combining Theorem~\ref{theorem::time::apxout} and Theorem~\ref{theorem::time::apxpoly} ensures that, for a fixed number~$p$ of parameters, Algorithm~\ref{algo::apx} runs in time $\bigO{\left(\poly(\langle \Pi \rangle,\varepsilon^{-1},T_{\oracle})\right)}$.
Therefore, if~$\oracle_\alpha$ runs in polynomial time, so does Algorithm~\ref{algo::apx}.

Furthermore, if the oracle~$\oracle_\alpha$ is provided through a conventional (non-parametric) F(PTAS), we can pick an appropriate $\varepsilon'>0$ such that $1+\varepsilon'= \sqrt{1+\varepsilon}$ and run Algorithm~\ref{algo::apx} with~$\varepsilon'$ as input and $\oracle_{(1+\varepsilon')}$ as oracle\footnote{%
    Note that~$\varepsilon'$ might not be computable or might not have polynomial encoding length.
    In that case, some appropriate rounding must be applied.%
}.
By Theorem~\ref{theorem::algo::apxGuarantee}, this results in a $(1+\varepsilon')^2=(1+\varepsilon)$-approximation set, and by Theorems~\ref{theorem::time::apxout} and~\ref{theorem::time::apxpoly}, it has the running time properties of an (F)PTAS\@.

\medskip

In the remainder of this section, we prove Theorem~\ref{theorem::time::apxpoly}.
This is done in two steps:
In the first step, we consider a polyhedral subset~$\Lambda^*\subseteq\Lambda$.
For this subset, we show that the number of solutions that are added to~$S$ after calls to~$\oracle_{\alpha}$ with parameter vectors from~$\Lambda^*$ is in $\bigO{\left(\poly(\langle \Pi \rangle,\varepsilon^{-1})\right)}^{\left|\vertx(\Lambda^*)\right|}$.
While the choice of~$\Lambda^*=\Lambda$ then provides a first bound on~$\smax$, this bound is superpolynomial in the instance size since $\left|\vertx(\Lambda^*)\right| \in \Theta{\left(m^{\lfloor\frac{p}{2}\rfloor}\right)}$ in the worst-case~\cite{ziegler1995LecturesPolytopes}.
In the second step, we use the concept of \emph{triangulation} from computational geometry to subdividen $\Lambda$ into a polynomial number of ($p$-dimensional) \emph{simplices} (see Figure~\ref{figure::algo::triangulation}).
A ($p$-dimensional) simplex is a $p$-dimensional polyhedron with exactly $p+1$~vertices.
If~$\Lambda^*$ is a simplex, the bound from above becomes $\bigO{\left(\poly(\langle \Pi \rangle,\varepsilon^{-1})\right)}^{p+1}$.
We then only have to count the number of simplices and multiply the bound by this number.

A special class of polyhedra is needed in this proof:
For $n\in\mathbb{N}_{>0}$, a \emph{hyperrectangle} in~$\mathbb{R}^n$ is a polyhedron $\{x\in\mathbb{R}^n:l\leq x \leq u\}$ defined by two points $l,u\in\mathbb{R}^n$, where $l\leq u$.
For $l\geq 0$ and $\varepsilon>0$, a \emph{$(1+\varepsilon)$-hyperrectangle} is a special hyperrectangle where $l\in\mathbb{R}^n_{\geq0}$ and $u=(1+\varepsilon)\cdot l$.

\medskip

We begin with the first bound.
Let $\Lambda^*\subseteq\Lambda$ be a $p$-dimensional polytope with vertices $v^1,\ldots,v^k\in\vertx(\Lambda)$.
Note that this implies $\langle v^i\rangle\in \bigO{\left(\poly\langle \Pi \rangle\right)}$ for $i=1,\ldots,k$~\cite{grotschel.lovasz.ea1988RationalPolyhedra}.
For each~$x$ that is added to~$S$ at some point in Algorithm~\ref{algo::apx}, let~$\lambda(x)$ be the parameter vector such that~$x$ was returned byn $\oracle_\alpha(\lambda(x))$ in Step~9 (or in Step~3).
We define the set $S(\Lambda^*)\coloneqq \{x \in \Smax: \lambda(x)\in\Lambda^*\}$ as the set of all solutions in~$\Smax$ that were returned by oracle calls for parameter vectors from~$\Lambda^*$.

We now map our set of solutions $\orsols$ into~$\mathbb{R}_{\geq0}^k$.
For this mapping, we can use a special subdivision of~$\mathbb{R}_{\geq0}^k$ to obtain the cardinality bound.

Let the matrix~$C\in\mathbb{Q}^{k\times p+1}$ be defined as $C\coloneqq{\left(c^1,\ldots,c^{k}\right)}^\T$, with $c^i=(v^i,1)$ for $i=1,\ldots,k$.
Now, consider the set
\[
Y_C \coloneqq \left\{ C f(x)\in\mathbb{R}^k: x\in \orsols\right\}.
\]
Two important properties hold for every~$y\in Y_C$.
First, since $y_i={c^i}^\T f(x)=F_{v^i}(x)\geq 0$ for some~$x\in \orsols$ and every $i\in\{1,\ldots,k\}$, it holds that~$y\geq0$.
Second, for $\rho\coloneqq \max_{y\in Y_C}\langle y_{i=1,\ldots,k} \rangle$, it holds that $\rho\in \bigO(\poly\langle \Pi\rangle)$: $x$ and~$c^i$ have encoding length in $\bigO{\left(\poly\langle \Pi \rangle\right)}$, and~$f$ is polynomial-time computable.

For a set with these two properties, it is possible to construct a set~$\subdi$ of $(1+\varepsilon)$-hyperrectangles with cardinality $\left|\subdi\right|$ in $\bigO{\left({\left(\nicefrac{\rho}{\varepsilon}\right)}^{k}\right)}$, and every~$y\in Y_C$ is contained in at least one $(1+\varepsilon)$-hyperrectangle~$\subdielem$ of~$\subdi$.
For brevity, we omit a description of $\subdi$, see~\cite{papadimitriou.yannakakis2000ApproximabilityTradeoffsOptimal,glasser.reiwiesner2010ApproximabilityHardness,nemesch.ruzika.ea2025PolynomialTimeInnerApproximation} for a detailed construction.

\begin{lemma}\label{lemma::apx::gridarg}
    For every $\subdielem\in\subdi$, there is at most one~$x\in S(\Lambda^*)$ such that $Cf(x)\in \subdielem$.
\end{lemma}
\begin{proof}
    Assume for the sake of contradiction that there are two solutions~$x,x^*\in S(\Lambda^*)$ with $C f(x)\in \subdielem$, $C f(x^*)\in \subdielem$, and $x\neq x^*$.

    By construction, we have $\lambda(x^*)\in\Lambda^*$.
    Then, there exists $\ell\in\mathbb{R}^k_{\geq0}$ with $\|\ell\|_1=1$ such that $\sum_{i=1}^k \ell_i v^i=\lambda(x^*)$.
    Thus, for  $F_{\lambda(x^*)}(x)$ we have
    \begin{alignat*}{1}
        \ell^\T C f(x) & = \sum_{i=1}^k \ell_i {c^i}^\T f(x) = \sum_{i=1}^k \ell_i \left( \sum_{j=1}^p v^i_j f_j(x) +f_{p+1}(x) \right) \\
        & =  \sum_{j=1}^p \left( \sum_{i=1}^k \ell_i   v^i_j \right) f_j(x) +f_{p+1}(x) = \sum_{j=1}^p {\lambda_i(x^*)} f_j(x) +f_{p+1}(x) \\
        & = F_{\lambda(x^*)}(x)
    \end{alignat*}
    and, similarly, $\ell^\T Cf(x^*)=F_{\lambda(x^*)}(x^*)$.

    Let $\subdivert\in\mathbb{R}^{k}_{\geq0}$ be the vertex defining the $(1+\varepsilon)$-hyperrectangle $\subdielem$, i.e., $\subdielem=\{z\in\mathbb{R}^{k}: \subdivert\leq z \leq (1+\varepsilon)\subdivert\}$.
    By our assumption, both $Cf(x)$ and $Cf(x^*)$ are in~$\subdielem$, which implies that $\subdivert\leq Cf(x)$ and $\subdivert\leq Cf(x^*)$.
    Since $Cf(x)$, $Cf(x^*)$, $g$, and $\ell$ have only non-negative entries, this yields $\ell^\T \subdivert \leq \ell^\T Cf(x)$ and $\ell^\T \subdivert \leq \ell^\T Cf(x^*)$.
    Similarly, we obtain $\ell^\T Cf(x) \leq (1+\varepsilon)\ell^\T \subdivert$ and $\ell^\T Cf(x^*) \leq (1+\varepsilon)\ell^\T \subdivert$.

    Without loss of generality, assume that~$x$ was already in~$S$ when~$x^*$ was added to~$S$ in Step~11 of Algorithm~\ref{algo::apx}.
    For~$x^*$ to be added to $S$, it must have held in Step~10 that $(1+\varepsilon)\cdot F_{\lambda(x^*)}(x^*)< F_{\lambda(x^*)}(x)$, which can also be written as $(1+\varepsilon)\cdot\ell^\T Cf(x^*) < \ell^\T Cf(x)$. 
    But then, $(1+\varepsilon)\cdot\ell^\T Cf(x^*) < \ell^\T Cf(x) \leq (1+\varepsilon)\ell^\T \subdivert$ follows.
    Dividing by $(1+\varepsilon)$ yields $\ell^\T Cf(x^*) < \ell^\T \subdivert$, which contradicts  $\ell^\T \subdivert \leq \ell^\T Cf(x^*)$.
\end{proof}

\begin{corollary}\label{corollary::apx::vertexbound}
    The cardinality of $S(\Lambda^*)$ is at most $\left|\subdi\right|\in\bigO{\left({\left(\nicefrac{\rho}{\varepsilon}\right)}^{k}\right)} = \bigO{\left({\left(\nicefrac{\rho}{\varepsilon}\right)}^{\left|\vertx(\Lambda^*)\right|}\right)}$.
\end{corollary}

Corollary~\ref{corollary::apx::vertexbound} provides a bound on the cardinality of~$S$ when choosing $\Lambda^*=\Lambda$.
But as $\left|\vertx(\Lambda)\right|\in \Theta(m^{\lfloor\frac{p}{2}\rfloor})$ in the worst case~\cite{ziegler1995LecturesPolytopes}, this bound can be exponential in $m$ and, thus, does not ensure a polynomial running time in Theorem~\ref{theorem::time::apxout}.

For the second bound, $\Lambda$ is subdivided into ($p$-dimensional) simplices.
Since every simplex has exactly $p+1$ vertices, the exponent in the bound from Corollary~\ref{corollary::apx::vertexbound} becomes a constant for each simplex in the subdivision.

A \emph{triangulation} of $\Lambda$ is a subdivision of $\Lambda$ into a finite number of simplices $\Lambda^1,\ldots,\Lambda^t$ such that, among other properties, $\cup_{i=1}^t \Lambda^i=\Lambda$ and $\vertx{\left(\Lambda^i\right)}\subseteq \vertx{\left(\Lambda\right)}$ for each $i\in\{1,\ldots,t\}$ (see~\cite{deloera.rambau.ea2010TriangulationsStructuresAlgorithms,lee.santos2017SubdivisionsTriangulations} for a more comprehensive introduction into triangulations).
An example is given in Figure~\ref{figure::algo::triangulation}.
From $\cup_{i=1}^t \Lambda^i=\Lambda$, it follows that $\cup_{i=1}^t S(\Lambda^i)=S$ and, thus, $\sum_{i=1}^t \left|S(\Lambda^i)\right|\geq \left|S\right|$.
At the same time, we can apply Corollary~\ref{corollary::apx::vertexbound} to each simplex~$\Lambda^i$ for $ i=1,\ldots,t$ to bound $\left|S(\Lambda^i)\right| \in \bigO{\left({\left(\nicefrac{\poly\langle\Pi\rangle}{\varepsilon}\right)}^{p+1}\right)}$.
Thus, $\left|S\right|\in \bigO{\left(t\cdot{\left(\nicefrac{\poly\langle\Pi\rangle}{\varepsilon}\right)}^{p+1}\right)}$.

It now only remains to bound the number~$t$ of simplices.
A known result for triangulations is that $t\in \bigO{\left(\left|\vertx (\Lambda)\right|^{\lceil\frac{p}{2}\rceil}\right)}$~\cite{lee.santos2017SubdivisionsTriangulations}.
Since also $\left|\vertx (\Lambda)\right|\in \bigO{\left(m^{\lfloor\frac{p}{2}\rfloor}\right)}$ holds~\cite{seidel1995UpperBoundTheorem}, we get the bound
$t\in \bigO{\left(m^{\lfloor\frac{p}{2}\rfloor\lceil\frac{p}{2}\rceil}\right)}$, which leads to the cardinality bound stated in Theorem~\ref{theorem::time::apxpoly}.

\begin{remark}
    Our proof of Theorem~\ref{theorem::time::apxpoly} induces a strategy that, in theory, can be used to generalize previously existing parametric approximation algorithms with their stricter assumptions (see Table~\ref{sota::table::results}).
    The idea is to compute the triangulation $\Lambda^1,\ldots,\Lambda^t$ explicitly,\footnote{%
    Computing an arbitrary triangulation for a given polytope can be done in polynomial time~\cite{lee.santos2017SubdivisionsTriangulations}, but obtaining a triangulation of minimum cardinality is NP-hard~\cite{below.loera.ea2000FindingMinimalTriangulations}.}
    construct matrices $C^1,\ldots, C^t$ for $\Lambda^1,\ldots,\Lambda^t$ analogously to the matrix~$C$ for~$\Lambda^*$, and then, for~$i=1,\ldots,t$, approximate the parametric optimization problem instance
    \[
        {\min \left\{\lambda^\T C^i f(x) : x\in X\right\}}_{\lambda\in\mathbb{R}_{\geq0}^{p+1}}.
    \]
    Without proof, we remark that all of these parametric optimization problem instances satisfy the stricter assumptions, and that the union of their $(1+\varepsilon)$-approximation sets is a $(1+\varepsilon)$-approximation set for the original problem instance~$\Pi$.
    However, in practice, such a procedure presumably runs slower than Algorithm~\ref{algo::apx}.
\end{remark}

\begin{figure}
    \centering
    \begin{tikzpicture}

    \def\cx{2.5}\def\cy{2.5}\def\r{2.5}
    \draw[draw=black!40, fill=black!7, very thick]
        ($(\cx,\cy)+(0:\r)$)
        \foreach \a in {45,90,135,180,225,270,315} { -- ($(\cx,\cy)+(\a:\r)$) }
        -- cycle;

    \node[yshift=5pt] at (\cx,\cy) {\large$\Lambda$};

    \def\cx{9}\def\cy{2.5}\def\r{2.5}

    \foreach \a/\i in {90/1,45/2,0/3,315/4,270/5,225/6,180/7,135/8} {
    \path coordinate (v\i) at ($(\cx,\cy)+(\a:\r)$);
    }

    \draw[draw=black!40, fill=black!7, very thick] (v1)--(v2)--(v3)--(v4)--(v5)--(v6)--(v7)--(v8)--cycle;

    \foreach \i in {3,4,5,6,7} {
    \draw[draw=black!40, very thick] (v1)--(v\i);
    }

    \foreach \a/\b/\c/\i in {1/2/3/6, 1/3/4/5, 1/4/5/4, 1/5/6/3,1/6/7/2, 1/7/8/1}
    {
    \coordinate (tmp\i) at ($(v\a)+(v\b)+(v\c)$);
    \node[yshift=5pt] at ($(0,0)!0.3333!(tmp\i)$) {$\Lambda^{\i}$};
    }
\end{tikzpicture}
    \caption{
        An example for a 2-dimensional parameter set $\Lambda$ and a possible triangulation.%
        \label{figure::algo::triangulation}
    }
\end{figure}
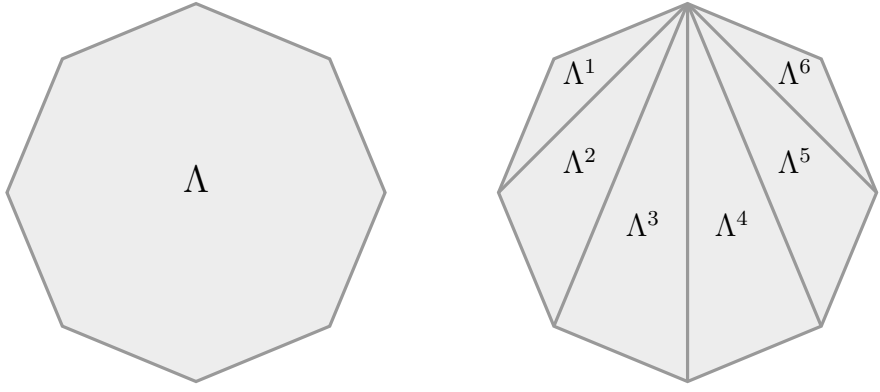

\section{Homogenization of the Parameter Set}\label{sec::homogenization}

Algorithm~\ref{algo::apx} can only be applied if the parameter set~$\Lambda$ is a polytope.
However, this is not necessarily the case for every instance of a parametric optimization problem.
In this section, we describe how we can transform every instance of a parametric optimization problem~$\Pi$ into a new instance~$\Pi^P$ of the same parametric problem, but with a polytope as parameter set.
Every $(1+\varepsilon)$-approximation set (or optimal solution set) for this new instance~$\Pi^P$ is a $(1+\varepsilon)$-approximation set (or optimal solution set) for the original instance~$\Pi$.

This section consists of two parts. We provide an overview before going into more detail.
In the first part, we construct an intermediary problem instance~$\Pi^\coneHull$ where the parameter set~$\coneHull$ is a cone, the so-called \emph{homogenization}~\cite{kaibel2011BasicPolyhedralTheory} of~$\Lambda$.
Furthermore, we show that, for every~$\varepsilon\geq0$ and~$\alpha\geq0$,
\begin{itemize}
    \item every $(1+\varepsilon)$-approximation set for~$\Pi^\coneHull$ is a $(1+\varepsilon)$-approximation set for~$\Pi$ (and vice versa), and
    \item $\oracle_\alpha$ can be used as an $\alpha$-approximation oracle for~$\Pi^\coneHull(w)$ for every~$w\in\coneHull$.
\end{itemize}

In the second part, we construct~$\Pi^P$ from~$\Pi^\coneHull$ by limiting the (conic) parameter set~$\coneHull$ to a subset~$P$ that is a polytope.
Meanwhile, the two aforementioned properties regarding approximation sets and the oracle~$\oracle_\alpha$ are conserved for~$\Pi^P$.

\medskip

We start the first part with an introduction into \emph{homogenizations}:
The homogenization~$\coneHull$ of a parameter set $\Lambda=\left\{\lambda\in\mathbb{R}^p: Q\lambda\geq d\right\}$ is the polyhedron
\[
    \coneHull \coloneqq\left\{(\lambda,\xi)\in\mathbb{R}^{p+1}: Q\lambda\geq \xi d,\;\xi\geq0\right\}.
\]

We now state some important properties of~$\coneHull$ --- for a more comprehensive introduction to homogenizations of polyhedra, we refer to~\cite{kaibel2011BasicPolyhedralTheory}.
A geometric interpretation is shown in Figure~\ref{figure::homogenizationIntro}.
The set~$\coneHull$ is a (polyhedral) cone and, since we assume~$\Lambda$ to be $p$-dimensional, $\coneHull$ is $(p+1)$-dimensional.
As notational shorthand, we define the function $\lwfunc\colon \mathbb{R}^{p+1}\to\mathbb{R}^{p}$, $\lwfunc(w)=\lwfunc(w_1,\ldots,w_{p+1})\coloneq(w_1,\ldots,w_p)$.
For the homogenization~$\coneHull$, there exist generators $V=\{v^1,\ldots,v^s\}\subseteq\mathbb{Q}^{p+1}$ and $R=\{r^1,\ldots,r^t\}\subseteq\mathbb{Q}^{p+1}$ such that $v^i_{p+1}=1$ for $i=1,\ldots,s$ and $r^j_{p+1}=0$ for $j=1,\ldots,t$, and such that
\[
    \coneHull = \cone\mleft(V\cup R\mright)
\]
    and
\[
    \Lambda = \conv\mleft(\lwfunc(V)\mright) + \cone\mleft(\lwfunc(R)\mright),
\]
where~$\lwfunc(V)$ and~$\lwfunc(R)$ are exactly the generators of~$\Lambda$~\cite{kaibel2011BasicPolyhedralTheory}.
As a consequence, $\coneHull$ has the same number of generators as~$\Lambda$.
Therefore, $s+t\in \bigO(m^{\left\lfloor\frac{p}{2}\right\rfloor})$ (see~\cite{seidel1995UpperBoundTheorem}).
Furthermore, the encoding length of each element of~$V$ and~$R$ can be bounded in the maximum encoding length of the generators of~$\Lambda$, which is in $\bigO(\poly\langle \Lambda\rangle)\in \bigO(\poly\langle \Pi \rangle)$ (see~\cite{grotschel.lovasz.ea1988RationalPolyhedra}).

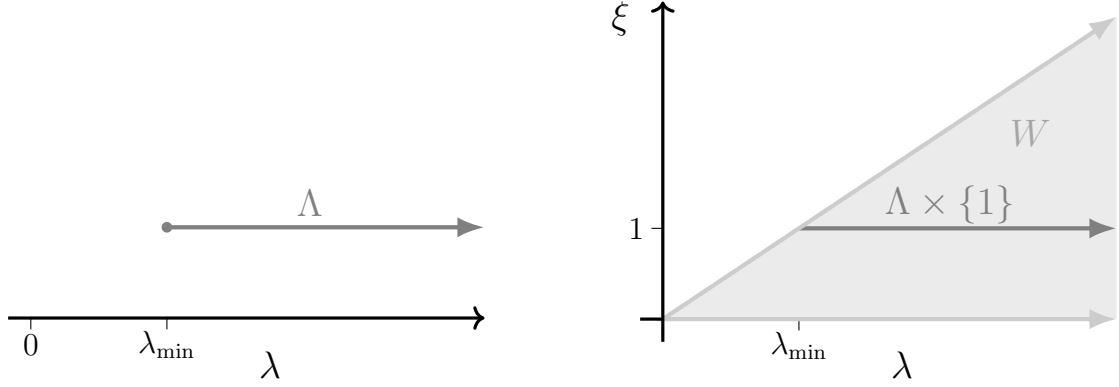
\begin{figure}
    \begin{minipage}[b]{0.475\textwidth}
        \centering
        \begin{tikzpicture}[%
  scale=0.6,
  oglambda/.style={draw=black!50, color=black!50, ultra thick, {-Latex}},
  Wbound/.style={draw=black!20, ultra thick, {Latex-Latex}},
  Winner/.style={draw=none, fill=black!8},
  Wnode/.style={color=black!35},
  tick/.style={thin},
]

  \draw[tick] (3,0) to (3,-0.25);
  \node[yshift=-10pt] at (3,0) {$\lambda_{\min}$};
  \draw[tick,opacity=0] (0,2) to (-0.25,2);
  \node[xshift=-10pt,opacity=0] at (0,2) {$1$};
  \draw[tick] (0,0) to (0,-0.25);
  \node[yshift=-10pt] at (0,0) {$0$};

  \draw[oglambda] (3,2) to node[pos=0.45, yshift=9pt] {\large$\Lambda$} (10,2);
  \filldraw [oglambda] (3,2) circle (2pt);

  \draw[very thick,->] (-0.5,0) to node[pos=0.55, below, yshift=-8pt] {\large$\lambda$} (10.0, 0);
  \draw[very thick,->,opacity=0] (0,-0.5)to node[pos=0.6, left, xshift=-8pt] {\large$\xi$} (0, 7.0);


\end{tikzpicture}
    \end{minipage}
    \hfill
    \begin{minipage}[b]{0.475\textwidth}
        \centering
        \begin{tikzpicture}[%
  scale=0.6,
  oglambda/.style={draw=black!50, color=black!50, ultra thick, {-Latex}},
  Wbound/.style={draw=black!20, ultra thick, {Latex-Latex}},
  Winner/.style={draw=none, fill=black!8},
  Wnode/.style={color=black!35},
  tick/.style={thin},
]

  \draw[tick] (3,0) to (3,-0.25);
  \node[yshift=-10pt] at (3,0) {$\lambda_{\min}$};
  \draw[tick] (0,2) to (-0.25,2);
  \node[xshift=-10pt] at (0,2) {$1$};

  \draw[Winner] (10,0) to (0,0) to (10,6.66666) to cycle;
  \node[Wnode] at (8.1,4.1) {\large$\coneHull$};

  \draw[oglambda] (3,2) to node[pos=0.475, yshift=9pt] {\large$\Lambda \times\{1\}$} (10,2);
  \draw[Wbound] (10,0) to (0,0) to (10,6.66666);

  \draw[draw=none] (-0.5,0) to node[pos=0.55, below, yshift=-8pt] {\large$\lambda$} (10, 0);
  \draw[very thick] (-0.5,0) to (0, 0);

  \draw[very thick,->] (0,-0.5)to node[pos=0.95, left, xshift=-8pt] {\large$\xi$} (0, 7.0);


\end{tikzpicture}
    \end{minipage}
    \caption{%
    Geometric interpretation of the homogenization~$\coneHull$ of the parameter set $\Lambda=\{\lambda \in \mathbb{R} : \lambda\geq\lambda_{\min}\}$.
    The set~$\Lambda$ is a $1$-dimensional polyhedron in~$\mathbb{R}$, and~$\coneHull$ is a $2$-dimensional cone in~$\mathbb{R}^2$.
    The intersection of~$\coneHull$ and the ``$\xi\!=\!1$''-plane is exactly the set $\Lambda\times\{1\}=\left\{(\lambda,1) : \lambda \in \Lambda\right\}$.
    For each point $(\lambda,1)\in \Lambda\times\{1\}$, $\coneHull$ contains all points $a\cdot(\lambda,1)$ with~$a\geq0$.
    In addition, $\coneHull$ contains the closure of the set of all these points.
    The sets~$V$ and~$R$ in this example $V=\{(\lambda_{\min},1)\}$ and $R=\{(1,0)\}$.%
    \label{figure::homogenizationIntro}}
\end{figure}

\begin{definition}
    For a parametric optimization problem instance $\Pi=(\Lambda,X,f)$,
    we call the parametric optimization problem instance
    \[
    \Pi^\coneHull \coloneq {\left\{\min_{x\in X} w^\T f(x) \right\}}_{w\in\coneHull}
    \]
    the \emph{parametric homogenization of~$\Pi$}.
\end{definition}

The parametric homogenization of~$\Pi$ can be seen as an extension of~$\Pi$.
For each~$\lambda\in\Lambda$, the point~$w=(\lambda,1)$ is in~$\coneHull$.
Since then $F_\lambda(x)=w^\T f(x)$, the (non-parametric) problem instances~$\Pi(\lambda)$ and~$\Pi^\coneHull(w)$ are equivalent for this choice of~$\lambda$ and~$w$.
In contrast, not every $w\in\coneHull$ has a matching element in $\Lambda$.
One can, however, show that both problems still share the same optimal solution sets and approximation sets.

\begin{theorem}\label{theorem::conehull::polysol}
    Let~$\varepsilon\geq0$ and let~$S\subseteq\orsols$ be a $(1+\varepsilon)$-approximation set for~$\Pi$.
    Then, for each~$w\in\coneHull$, the problem~$\Pi^\coneHull(w)$ is bounded and $(1+\varepsilon)$-approximated by at least one solution~$x$ from~$S$.
\end{theorem}
\begin{proof}
    Each~$w\in\coneHull$ can be written as
    \[
        w=\sum_{i=1}^s\ell_i v^i + \sum_{j=1}^t k_j r^j
    \]
    for some $\ell\in\mathbb{R}_{\geq0}^s$, $k\in\mathbb{R}_{\geq0}^t$.
    We prove the statement of Theorem~\ref{theorem::conehull::polysol} separately for the two cases: a) $\ell\neq0$ and b) $\ell=0$.

    \medskip

    Case~a); $\ell\neq0$:
    We scale both vectors~$\ell$ and~$k$ by~$\nicefrac{1}{\|\ell\|_1}$ and obtain $\frac{1}{\|\ell\|_1}\sum_{i=1}^s\ell_i v^i\in\conv(V)$ and  $\frac{1}{\|\ell\|_1}\sum_{j=1}^t k_j r^j\in\cone(R)$.
    Let~$\hat{w}$ be the scaled vector~$\frac{w}{\|\ell\|_1}$.
    Since $\hat{w}\in \conv(V) + \cone(R)$, we have $\lwfunc(\hat{w})\in \conv(\lwfunc(V)) + \cone(\lwfunc(R))$ and, thus, $\lwfunc(\hat{w})\in\Lambda$.
    Also, $\hat{w}_{p+1}=\frac{1}{\|\ell\|_1}\sum_{i=1}^s \ell_i v^i_{p+1}=\frac{1}{\|\ell\|_1}\sum_{i=1}^s(\ell_i\cdot 1)=1$.
    Both the problems~$\Pi{\left(\lwfunc(\hat{w})\right)}$ and~$\Pi^\coneHull(\hat{w})$ can be written as $\min_{x\in X}\hat{w}^\T f(x)$ and are, therefore, equivalent.
    Let the solution~$x$ be a $(1+\varepsilon)$-approximation for~$\Pi{\left(\lwfunc(\hat{w})\right)}$, then it is also a $(1+\varepsilon)$-approximation for~$\Pi^\coneHull(\hat{w})$.
    Scaling~$\hat{w}$ by a positive factor only affects the optimal objective function value, but not which solutions are $(1+\varepsilon)$-approximations for~$\Pi^\coneHull(\hat{w})$.
    Consequently, $x$ is also a $(1+\varepsilon)$-approximation for~$\Pi^\coneHull(w)$.
    Therefore, $\Pi^\coneHull(w)$ is bounded and~$x$ is an $(1+\varepsilon)$-approximation.

    \medskip

    Case~b); $\ell=0$:
    Here, let~$v$ be an arbitrary element of~$V$.
    For each~$q>0$, we define
    \[
        w(q)\coloneqq v+q\cdot w.
    \]
    By construction, for every~$q\geq0$, $\lwfunc(w(q))\in\Lambda$ holds, $\Pi^\coneHull(w(q))$ is bounded, and~$\Pi^\coneHull(w(q))$ has a $(1+\varepsilon)$-approximate solution in $S$ (see Case~a) discussed above).

    Let $\hat{w}(q)\coloneqq\nicefrac{1}{q}\cdot w(q)$.
    For every~$q>0$, the problem~$\Pi^\coneHull(\hat{w}(q))$ is the problem~$\Pi^\coneHull(w(q))$ with the objective function scaled by the positive factor~$\nicefrac{1}{q}>0$.
    Therefore, $\Pi^\coneHull(\hat{w}(q))$ is also $(1+\varepsilon)$-approximated by a solution in $S$.
    With increasing~$q$, the vector~$\hat{w}(q)= \nicefrac{1}{q}\cdot v + w$ converges to the limit~$w$ and, for each~$x\in X$, ${\hat{w}(q)}^\T f(x)=\nicefrac{1}{q}\cdot v^\T f(x) + w^\T f(x)$ then converges to the limit $w^\T f(x)$.
    Since~$S$ is finite, there must be a~$q^*\geq0$ and a solution~$x^*\in S$ such that $x^*\in\arg\min_{x\in S}\hat{w}(q)^\T f(x)$ for all $q\geq q^*$.
    Then, $x^*$ also is a $(1+\varepsilon)$-approximation for every problem $\Pi^\coneHull(\hat{w}(q))$ with $q\geq q^*$.

    Now, for the sake of contradiction, assume that there is an optimal solution~$x'\in X$ of~$\Pi^\coneHull(w)$ such that $(1+\varepsilon)w^\T f(x')<w^\T f(x^*)$.
    But then
    \[
        \lim_{q\rightarrow\infty} (1+\varepsilon){\hat{w}(q)}^\T f(x')
        = (1+\varepsilon)w^\T f(x')
        < (1+\varepsilon)w^\T f(x^*)
        = \lim_{q\rightarrow\infty} (1+\varepsilon){\hat{w}(q)}^\T f(x^*),
    \]
    which is a contradiction to~$x^*$ being $(1+\varepsilon)$-approximate for every problem~$\Pi^\coneHull(\hat{w}(q))$ with~$q\geq q^*$.
    Therefore, $x^*$ must be a $(1+\varepsilon)$-approximate solution of~$\Pi^\coneHull(w)$.
\end{proof}

By setting $\varepsilon=0$, we can also infer the following property regarding optimal solutions in the approximation setting from the proof of Theorem~\ref{theorem::conehull::polysol}:
\begin{corollary}\label{corollary::conehull::positive}
    If $\min_{x\in X}F_\lambda(x)\geq0$ for every $\lambda\in\Lambda$, then $\min_{x\in X}w^\T f(x)\geq0$ for every $w\in\coneHull$.
\end{corollary}

Theorem~\ref{theorem::conehull::polysol} ensures that, if a $(1+\varepsilon)$ approximation set for~$\Pi$ exists, then there also exists a $(1+\varepsilon)$-approximation set for~$\Pi^{\coneHull}$.
The converse is also true:
For each~$\lambda$ in $\Lambda$, the problems~$\Pi(\lambda)$ and~$\Pi^\coneHull(w)$ are equivalent if~$w=(\lambda,1)$, which is a vector in~$\coneHull$.
Hence, every $(1+\varepsilon)$-approximation set for~$\Pi^\coneHull$ then also is a $(1+\varepsilon)$-approximation set for~$\Pi$.

\medskip

Next, we discuss how to use the oracle~$\oracle_\alpha$ from the original problem~$\Pi$ for the parametric homogenization~$\Pi^\coneHull$.
In practice, many oracles can directly be used without any changes since the parametric value functions~$F_\lambda(x)$ of~$\Pi$ and $w^\T f(x)$ of~$\Pi^\coneHull$ are both just linear combinations of the components of~$f$.
For fixed values of~$\lambda$ or~$w$, the problems~$\Pi(\lambda)$ and~$\Pi^\coneHull(w)$ are both conventional (non-parametric) optimization problems.
Nearly all exact and approximation algorithms that are used in practice, e.g., MIP solvers or the approximation algorithms in~\cite{vazirani2001ApproximationAlgorithms}, can deal with arbitrary linear objective functions.

Even if, for some reason, the oracle~$\oracle_\alpha$ only works for inputs~$\lambda\in\Lambda$, it can be used to simulate an $\alpha$-approximation oracle for~$\Pi^\coneHull$:

\begin{proposition}\label{proposition::conehull::oracle}
    For every~$w\in\coneHull\cap \mathbb{Q}^{p+1}$, there is at least one~$\lambda\in\Lambda$ with $\langle \lambda\rangle\in \bigO(\poly\langle \Pi,w\rangle)$ such that~$\oracle_\alpha(\lambda)$ returns an $\alpha$-approximate solution of~$\Pi^\coneHull(w)$.
\end{proposition}

Since the formal proof is rather technical and, most likely, of purely theoretical value, we only sketch the necessary steps here.
As in the proof of Theorem~\ref{theorem::conehull::polysol}, the two cases~a) and~b) need to be considered:
In Case~a), $w$ can be scaled directly to some~$\hat{w}\in\coneHull$ such that there is a~$\lambda\in\Lambda$ where~$\Pi^\coneHull(\hat{w})$ is equivalent to~$\Pi(\lambda)$.
In Case~b), we construct a vector~$\hat{w}(q)$ that converges to~$w$ and where corresponding values in~$\Lambda$ exist.
It can then be shown that there exists a~$q^*>0$ with $\langle q^* \rangle\in\bigO{\left(\poly(\langle I,w\rangle)\right)}$ such that a $(1+\varepsilon)$-approximation~$x\in\orsols$ for~$\Pi^\coneHull(\hat{w}(q^*))$ is also a $(1+\varepsilon)$-approximation for~$\Pi^\coneHull(w)$.
This~$q^*$ can be determined in polynomial time.

\subsection{Obtaining a Polytope}\label{sec::polytope}

Since the parameter set~$\coneHull$ of~$\Pi^\coneHull$ is a cone and not a polytope, Algorithm~\ref{algo::apx} is not applicable to the problem~$\Pi^\coneHull$.
We now discuss how to obtain a new parametric problem~$\Pi^P$ by reducing~$\coneHull$ into a polytope and using this polytope as the parameter set.
Then, every $(1+\varepsilon)$-approximation set for~$\Pi^P$ is a $(1+\varepsilon)$-approximation set for~$\Pi^\coneHull$ and, thus, also a $(1+\varepsilon)$-approximation set for the original problem~$\Pi$.
We distinguish two cases: either~$\coneHull$ is \emph{pointed}, i.e.\ $\coneHull$ does not contain a line, or~$\coneHull$ is not pointed.

\medskip

\noindent
We start with the case where~$\coneHull$ is pointed.
Note that the strategy we use here is not completely novel, an analogous strategy is already used by \citeauthor{hamel.lohne.ea2014BensonTypeAlgorithms}~\cite{hamel.lohne.ea2014BensonTypeAlgorithms} for vector optimization problems.
The \emph{dual cone} of~$\coneHull$ is the cone
\begin{align*}
    D_\coneHull \coloneqq{} &  \left\{u\in\mathbb{R}^{p+1}: u^\T w\geq 0 \text{ for all } w\in \coneHull\right\} \\
    ={} & \left\{u\in\mathbb{R}^{p+1}: u^\T w\geq 0 \text{ for all } w \in V \cup R \right\}.
\end{align*}
We choose a vector~$u$ from the interior of~$D_\coneHull$ such that $\langle u\rangle \in \bigO(\poly\langle\Pi\rangle)$ and $u_{p+1}=1$.
A vector $u$ satisfying these conditions always exists:
Since~$\coneHull$ is pointed, $D_\coneHull$ is full dimensional~\cite{schrijver1998TheoryLinearInteger}. 
Therefore, there must be a vector~$\hat{u}$ in the interior of~$D_\coneHull$ with $\hat{u}_i\neq0$ for $i=1,\ldots,p+1$.
Since~$w_{p+1}\geq0$ for all $w\in \coneHull$, the point $(\hat{u}_1,\ldots,\left|\hat{u}_{p+1}\right|)$ must be in~$D_\coneHull$ too, and we can choose~$u$ as $u\coloneqq\frac{1}{\left|\hat{u}_{p+1}\right|}(\hat{u}_1,\ldots,\left|\hat{u}_{p+1}\right|)$.
Furthermore, since the rows in an $\mathcal{H}$-representation of~$D_\coneHull$ are exactly the generators~$V$ and~$R$ from the $\mathcal{V}$-representation of $\coneHull$~\cite{schrijver1998TheoryLinearInteger}
, it can be easily ensured that $\langle u\rangle \in \bigO(\poly\langle\Pi\rangle)$.
Note that, since~$u$ is from the interior of~$D_\coneHull$, it holds that $u^\T w > 0$ for every $w\in\coneHull \setminus \{0\}$.

We now use~$u$ to construct a hyperplane that intersects~$\coneHull$ such that the intersection is a polytope and can be used as our new parameter set.
More formally, $\coneHull$ is intersected with the hyperplane $\{(\lambda,\xi)\in\mathbb{R}^{p+1}: u^\T(\lambda,\xi)=1\}$, resulting in the set
\[
    P\coloneqq \left\{(\lambda,\xi)\in\mathbb{R}^{p+1}:Q\lambda\geq \xi d,\;\xi\geq0,\;u^\T(\lambda,\xi)=1\right\}.
\]
An example of the construction of~$P$ is given in Figure~\ref{figure::homogenizationPolytope}.

\begin{lemma}
    The set~$P$ is a polytope.
\end{lemma}
\begin{proof}
    We prove the claim by contradiction.
    Assume that~$P$ is not a polytope.
    Then there is a ray $r\in\mathbb{R}^{p+1}\setminus\{0\}$ such that $v+ar\in P$ for every $a>0,v \in P$.
    Since $P\subseteq \coneHull$, $r$ must be a ray in $\coneHull$ too, and because~$\coneHull$ is a cone, $r$ itself must be in~$\coneHull$.
    Then, $u^\T r>0$ holds (recall that~$u$ is in the interior of~$D_\coneHull$).
    For every $a>0,v \in P$, it follows that $u^\T(v+ar)=1+a u^\T r>1$ and, thus, $v+ar\notin P$, which is a contradiction.
\end{proof}

Furthermore, $P$ is $p$-dimensional, since it is the intersection of a ($p$-dimensional) hyperplane and a $(p+1)$-dimensional polyhedron, where the hyperplane intersects the interior of the polyhedron.

Let~$\Pi^P$ be the $(p+1)$-parametric optimization problem ${\left\{\min_{x\in X} w^\T f(x)\right\}}_{w\in P}$, i.e., the problem~$\Pi^\coneHull$ with~$P$ instead of~$\coneHull$ as the parameter set.
By choice of~$u$, for every~$w\in W\setminus\{0\}$, there is exactly one~$w'\in P$ with $aw=w'$ for some $a>0$.
If a solution~$x\in\orsols$ is a $(1+\varepsilon)$-approximation for~$\Pi^P(w')$, it is also a $(1+\varepsilon)$-approximation for~$\Pi^\coneHull(w)$.
Therefore, every $(1+\varepsilon)$-approximation set for~$\Pi^P$ is a $(1+\varepsilon)$-approximation for~$\Pi^{\coneHull}$ and, thus, for the original problem~$\Pi$ itself.

Due to the equality $u^\T(\lambda,\xi)=1$, it is possible to substitute~$\xi$ with $1-\sum_{i=1}^p u_i\lambda_i$.
By doing so, $\Pi^P$ becomes a $p$-parametric problem again.
Furthermore, the number of generators for the parameter set~$P$ is the same as for~$\coneHull$ and, thus, as for the original parameter set~$\Lambda$.

\medskip

Next, we look at the case where~$\coneHull$ contains a line.
The intersection of~$\coneHull$ and a hyperplane as in the previous case would be unbounded here.
Instead, we construct the polyhedron
\[
P \coloneqq \conv({V \cup R})
\]
and use it to construct the problem~$\Pi^P$ as the problem ${\left\{\min_{x\in X} w^\T f(x)\right\}}_{w\in P}$.
For every $w\in\coneHull\setminus\{0\}$, there is at least one~$w'\in P$ with $aw=w'$ for some $a>0$.
Again, every $(1+\varepsilon)$-approximation for~$\Pi^P$ is also an $(1+\varepsilon)$-approximation for~$\Pi^{\coneHull}$ and for the original problem~$\Pi$.
While the number of generators for~$P$ is the same as for~$\coneHull$ and, thus, $\Lambda$, it is a $(p+1)$-dimensional polytope.
Hence, we cannot substitute~$\xi$ as seen above.

\medskip

In both cases, the problem~$\Pi^P$ can be constructed from the problem~$\Pi$ in polynomial time by following the steps outlined in this section.
Algorithm~\ref{algo::apx} can then be used to compute an optimal solution set or an approximation set.
Since $\langle\Pi^P\rangle\in\bigO{\left(\poly\langle \Pi \rangle\right)}$ and the number of parameters increases by at most one, Algorithm~\ref{algo::apx} computes an $(1+\varepsilon)\alpha$-approximation set in $\bigO{\left(\poly(\langle\Pi\rangle,\varepsilon^{-1})\right)}$ as long as a polynomial-time oracle~$\oracle_\alpha$ can be provided (see Theorems~\ref{theorem::time::apxout} and~\ref{theorem::time::apxpoly}).

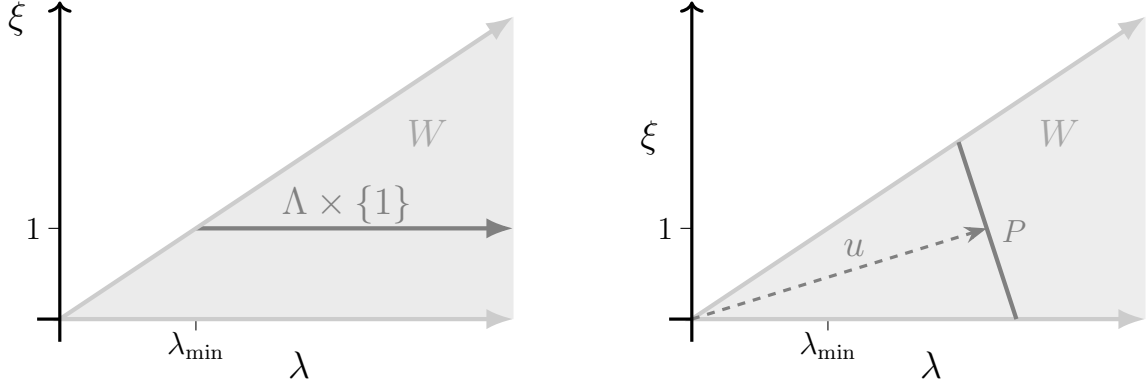
\begin{figure}
    \begin{minipage}[b]{0.475\textwidth}
        \centering
        \begin{tikzpicture}[%
  scale=0.6,
  oglambda/.style={draw=black!50, color=black!50, ultra thick, {-Latex}},
  Wbound/.style={draw=black!20, ultra thick, {Latex-Latex}},
  Winner/.style={draw=none, fill=black!8},
  Wnode/.style={color=black!35},
  tick/.style={thin},
]

  \draw[tick] (3,0) to (3,-0.25);
  \node[yshift=-10pt] at (3,0) {$\lambda_{\min}$};
  \draw[tick] (0,2) to (-0.25,2);
  \node[xshift=-10pt] at (0,2) {$1$};

  \draw[Winner] (10,0) to (0,0) to (10,6.66666) to cycle;
  \node[Wnode] at (8.1,4.1) {\large$\coneHull$};

  \draw[oglambda] (3,2) to node[pos=0.475, yshift=9pt] {\large$\Lambda \times\{1\}$} (10,2);
  \draw[Wbound] (10,0) to (0,0) to (10,6.66666);

  \draw[draw=none] (-0.5,0) to node[pos=0.55, below, yshift=-8pt] {\large$\lambda$} (10, 0);
  \draw[very thick] (-0.5,0) to (0, 0);

  \draw[very thick,->] (0,-0.5)to node[pos=0.95, left, xshift=-8pt] {\large$\xi$} (0, 7.0);


\end{tikzpicture}
    \end{minipage}
    \hfill
    \begin{minipage}[b]{0.475\textwidth}
        \centering
        \begin{tikzpicture}[%
  scale=0.6,
  Wbound/.style={draw=black!20, ultra thick, {Latex-Latex}},
  Winner/.style={draw=none, fill=black!7},
  Wnode/.style={color=black!35},
  tick/.style={thin},
  uvector/.style={draw=black!50, color=black!50, dashed, very thick, {-Stealth}},
  Wu/.style={draw=black!50, color=black!50, ultra thick},
]

  \draw[tick] (3,0) to (3,-0.25);
  \node[yshift=-10pt] at (3,0) {$\lambda_{\min}$};
  \draw[tick] (0,2) to (-0.25,2);
  \node[xshift=-10pt] at (0,2) {$1$};

  \draw[Winner] (10,0) to (0,0) to (10,6.66666) to cycle;
  \node[Wnode] at (8.1,4.1) {\large$\coneHull$};

  \draw[Wbound] (10,0) to (0,0) to (10,6.66666);
  \draw[uvector,shorten >=0.0pt] (0,0) to node[pos=0.55, yshift=8pt] {\large$u$} (6.5,2);
  \draw[Wu] (7.15385,0) to node[pos=0.5, xshift=10pt] {$P$} ($(7.15385,0)!1.95!(6.5,2)$);

  \draw[draw=none] (-0.5,0) to node[pos=0.55, below, yshift=-8pt] {\large$\lambda$} (10, 0);
  \draw[very thick] (-0.5,0) to (0, 0);

  \draw[very thick,->] (0,-0.5)to node[pos=0.6, left, xshift=-8pt] {\large$\xi$} (0, 7.0);


\end{tikzpicture}
    \end{minipage}
    \caption{%
    Geometric interpretation of the construction of~$P$ (for the example from Figure~\ref{figure::homogenizationIntro}):
    The new parameter set~$P$ is obtained by intersecting the cone~$\coneHull$ with the hyperplane $\{(\lambda,\xi)\in\mathbb{R}^{p+1}: u^\T(\lambda,\xi)=1\}$.
    Then, $P$ is a $1$-dimensional polytope in~$\mathbb{R}^2$.%
    \label{figure::homogenizationPolytope}}
\end{figure}

\section{Non-Approximable Maximization Problems}\label{sec::maxNotApx}

In Assumption~\ref{assumption::general}, we assume that~$F_\lambda(x)\geq0$ for all~$x\in X$ and all~$\lambda\in\Lambda$.
This ensures that $F^*(\lambda)\geq0$ holds for all~$\lambda\in\Lambda$ --- a necessary condition for approximation:
If $F^*(\lambda)<0$ for some~$\lambda\in\Lambda$, then $(1+\varepsilon) \cdot F^*(\lambda)<F^*(\lambda)\leq F_\lambda(x)$ for all~$x\in X$, and no solution~$x$ can approximate~$\Pi(\lambda)$.
While, for parametric minimization problems, $F_\lambda(x)\geq0$ for all~$x\in X$ and all~$\lambda\in \Lambda$ is equivalent to $F^*(x)\geq0$ for all~$\lambda\in \Lambda$, the same does not hold for parametric maximization problems.
In the maximization case, there might be a solution~$x$ such that $F_\lambda(x)<0$ for some~$\lambda\in\Lambda$ while, simultaneously, $F^*(\lambda)\geq 0$.
This raises the following question: If the assumption $F_\lambda(x)\geq0$ for all~$x\in X$ and all~$\lambda\in\Lambda$ is replaced by the weaker assumption that $F^*(\lambda)\geq 0$ for all~$\lambda\in\Lambda$, does Algorithm~\ref{algo::apx} still find an approximation set in polynomial time (with an appropriate oracle~$\oracle_\alpha$)?

Due to the aforementioned equivalency in the minimization case, we need to consider this question only for parametric maximization problems.
Recall that, in the maximization case, $0<\varepsilon<1$ and a solution~$x$ is a $(1-\varepsilon)$-approximation for~$\Pi(\lambda)$ if $F_\lambda(x)\geq (1-\varepsilon)\cdot F^*(\lambda)$.
Even under the weaker assumption, Algorithm~\ref{algo::apx} is finite and outputs a $(1-\varepsilon)\alpha$-approximation set.
This follows directly from the proof of Theorem~\ref{theorem::algo::apxGuarantee}, which only requires $F^*(\lambda)\geq0$.
This leaves the polynomial running time.
The proof of Theorem~\ref{theorem::time::apxpoly}, however, requires the stricter assumption of $F_\lambda(x)\geq0$ for all~$x\in X$ and all~$\lambda\in\Lambda$.

\medskip

We now show that, under the weaker assumption that $F^*(\lambda)\geq 0$ for all~$\lambda\in\Lambda$, there exist parametric maximization problems with families of  instances where the cardinality of every $(1-\varepsilon)$-approximation set grows superpolynomial in the encoding size.
Then, every $(1-\varepsilon)$-approximation algorithm must run in superpolynomial time, regardless of whether a polynomial-time oracle~$\oracle_\alpha$ exists.
In fact, such problems also exist under the stronger assumption that $F^*(\lambda)\geq K$ for an arbitrary constant~$K\geq0$.
This result is stated formally in Theorem~\ref{theorem::maxNotApx}, and the rest of this section is used to prove it.

\begin{theorem}\label{theorem::maxNotApx}
    Let $K\geq0$ and $0<\varepsilon<1$.
    For every~$k\in\mathbb{N}$, there exists a $2$-parametric maximization linear program~$\Pi$ such that $F^*(\lambda)\geq K$ for all~$\lambda\in\Lambda$, $\langle \Pi \rangle \in \bigO(\poly(k))$, and every $(1-\varepsilon)$-approximation set has cardinality in~$\Omega(2^k)$.
\end{theorem}

To prove Theorem~\ref{theorem::maxNotApx}, we start by modifying some results that, originally, were used in multi-objective optimization.
A bi-objective linear program is of the form
\begin{align*}\tag{$\multiSingleBase$}\label{problem::multiSingleBase}
    {\setlength{\arraycolsep}{1pt}
    \begin{array}{l r l}
        \max \quad & \multicolumn{2}{r}{\left({c^1}^\T x, {c^2}^\T x\right)} \\
        \mathrm{s.t.} & A x &\geq b, \\
        & x& \in\mathbb{R}^n
    \end{array}}
\end{align*}
for some $A\in\mathbb{Q}^{q\times n}$, $b\in\mathbb{Q}^q$, and $c^1,c^2\in\mathbb{Q}^n$.
We denote the set of feasible solutions by~$\Xkone$ and the set of images by~$\Ykone$.
Both sets are polyhedra.
By~$\YEkone$, we denote the set of non-dominated extreme points of~$\Ykone$, and by $N\coloneqq \left|\YEkone\right|$ its cardinality.
It is well-known that, for every $k>0$, the matrix~$A$ as well as the vectors~$b$, $c^1$, and~$c^2$ can be chosen such that $\langle \multiSingleBase \rangle\in \bigO(\poly(k))$, and $N\in\Omega(2^k)$.
Both~\cite{murty1980ComputationalComplexity} and~\cite{ruhe1988ComplexityResults} demonstrate how to construct such instances.
Due to the complexity of their constructions, we omit the descriptions here.
The additional properties that $\Ykone\subseteq\mathbb{R}_{\geq0}^2$ and $\YEkone\subseteq\mathbb{R}_{>0}^2$ can be ensured by standard techniques, e.g., by adding some offsets to the objectives and introducing auxiliary variables.
An illustration of the sets~$\Ykone$ and~$\YEkone$ is shown in Figure~\ref{figure::maxNotApx::multicrit::bi}.
For the rest of this section, we assume $k>0$ to be fixed and~\ref{problem::multiSingleBase} to have the aforementioned properties.

In the parametric setting, this bi-objective linear program corresponds to a $1$-parametric linear program with $\Lambda=\mathbb{R}_{\geq0}$, $X=\Xkone$, and $f(x)={({c^1}^\T x,{c^2}^\T x)}^\T$ (see Section~\ref{sec::prelims}).
This program can be $(1-\varepsilon)$-approximated, as it satisfies Assumption~\ref{assumption::general}.

We now transform~\ref{problem::multiSingleBase} into the following $3$-objective linear program:
\begin{align*}\tag{$\multiBiBase$}\label{problem::multiBiBase}
    {\setlength{\arraycolsep}{1pt}
    \begin{array}{l r l}
        \max \quad & \multicolumn{2}{r}{({c^1}^\T x,{c^2}^\T x,-\xi)} \\
        \mathrm{s.t.} & A x &\geq \xi b, \\
        & \multicolumn{2}{r}{0\leq\xi\leq1,}\\
        & (x,\xi)& \in\mathbb{R}^{n+1}.
    \end{array}}
\end{align*}
Here, we denote the set of feasible solutions by~$\Xktwo$, the set of images by~$\Yktwo$, and the set of non-dominated extreme points by~$\YEktwo$.
Clearly, $\langle \multiBiBase \rangle \in  \bigO(\langle\multiSingleBase\rangle)\subseteq \bigO(\poly(k))$.

\medskip

For every $\xi\in[0,1]$, the set
\[
\left\{\left({c^1}^\T x, {c^2}^\T x\right): A x =\xi b, x\in\mathbb{R}^n \right\}
\]
is equal to the set
\[
\left\{\xi \left({c^1}^\T x, {c^2}^\T x\right): A x = b, x\in\mathbb{R}^n \right\},
\]
we formally prove this in Lemma~\ref{lemma::maxNotApx::setsEqual} in the appendix.

This has the following effect: If~$\xi$ is fixed to~$1$, the structure of the problem described by~\ref{problem::multiBiBase} is identical to that of~\ref{problem::multiSingleBase}, and the non-dominated images are identical in the first two objectives.
Improving the third objective by decreasing the value of~$\xi$ simply scales down this set of non-dominated images until all objectives are zero.
Figure~\ref{figure::maxNotApx::multicrit} illustrates an example of both~$\Ykone$ and the resulting~$\Yktwo$.

For every feasible solution~$(x,\xi)$ with $0<\xi<1$, the solution~$(\frac{1}{\xi}x,1)$ is also feasible, and~$(x,\xi)$ can be written as the convex combination $(x,\xi)=(1-\xi)\cdot(0,0,0)+\xi\cdot(\frac{1}{\xi}x,1)$.
Therefore, all vertices of~$\Yktwo$ have either~$0$ or $-1$ (since the third objective is $-\xi$) as their third coordinate.
The set of non-dominated extreme points of Problem~\ref{problem::multiBiBase} then is
\[
\YEktwo = \left\{(y_1,y_2,-1): y\in \YEkone \right\} \cup \{(0,0,0)\}.
\]

From now on, $y^0=0$ denotes the all-zero image in~$\YEktwo$, and $y^1,\ldots,y^N$ denote the other images in~$\YEktwo$.
Furthermore, let $y^1,\ldots,y^N$ be ordered by their first objective value in increasing fashion, i.e., $y_{1}^1<\cdots<y_{1}^{N}$.
Since all these images are non-dominated, this implies that $y_{2}^1>\cdots>y_{2}^{N}$.
A visualization of this ordering is provided in Figure~\ref{figure::maxNotApx::multicrit}.
For each $i\in\{0,\ldots,N\}$, we let $x^i$ denote a feasible solution that maps to~$y^i$, and by $\XEktwo\coloneqq\{x^i:i=0,\ldots,N\}$ we denote the set containing these solutions.

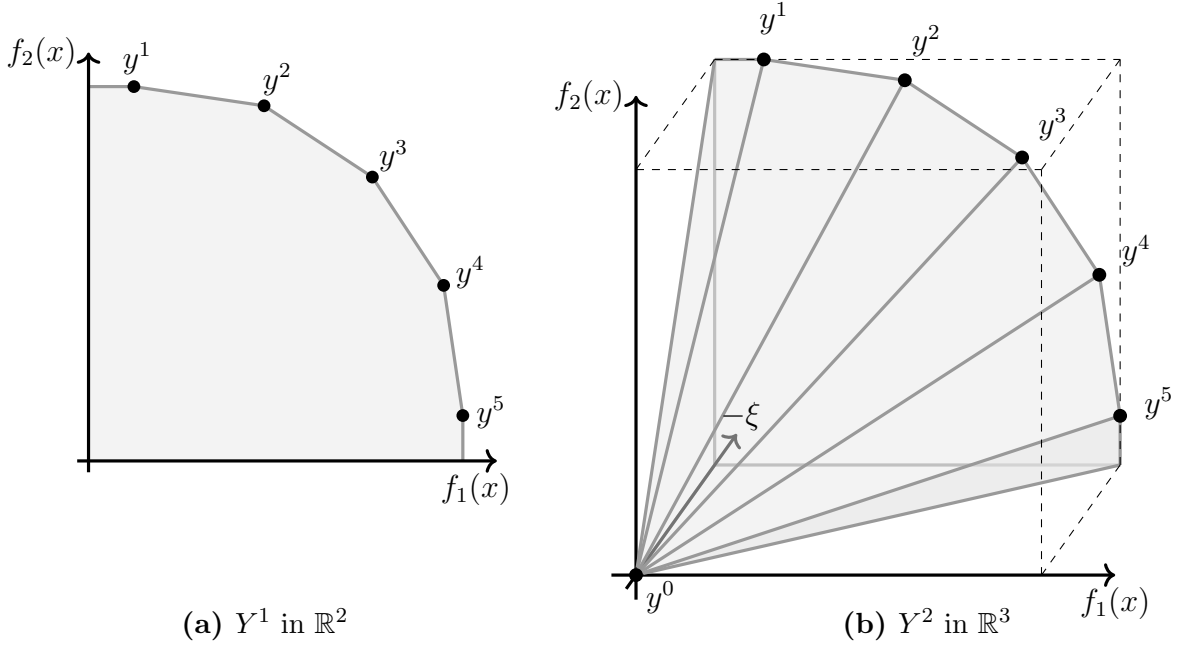
\begin{figure}
    \begin{minipage}[b]{0.4\textwidth}
        \centering
        \begin{tikzpicture}[scale=0.6,
    edgevisible/.style={draw=black!40, very thick},
    edgeinvisible/.style={draw=black!40, very thick, dashed},
    facevisible/.style={draw=none, fill=black!9, opacity=0.5},
    faceinvisible/.style={draw=none,},
    expoint/.style={fill=black},
    boxvisible/.style={thin, dashed},
    boxinvisible/.style={thin, dashed},]


    \node(y0) at (0.0,8.25) {};
    \node(y1) at (1.0,8.25) {};
    \node(y2) at (3.869,7.827) {};
    \node(y3) at (6.256,6.256) {};
    \node(y4) at (7.827,3.869) {};
    \node(y5) at (8.25,1.0) {};
    \node(y6) at (8.25,0.0) {};
    \node(y7) at (0.0,0.0) {};

    \draw[facevisible] (y0.center) to (y1.center) to (y2.center) to (y3.center) to (y4.center) to (y5.center) to (y6.center) to (y7.center) to cycle;
    \draw[edgevisible] (y0.center) to (y1.center) to (y2.center) to (y3.center) to (y4.center) to (y5.center) to (y6.center);

    \foreach \i in {1,2,3,4,5} {
        \fill[expoint] (y\i) circle (4pt);
    }

    \foreach \i in {1,2,3,4,5} {
        \draw[draw=none] (y\i.center) to node[pos=-0.075] {$y^{\i}$} (0,0);
    }

	\draw[very thick,->] (-0.25,0) to node[pos=0.95, below] {$f_1(x)$} (9.0, 0);
	\draw[very thick,->] (0,-0.25)to node[pos=1.0, left] {$f_2(x)$} (0, 9.0);

\end{tikzpicture}
        \vspace*{1.0cm}
    \end{minipage}
    \hfill
    \begin{minipage}[b]{0.55\textwidth}
        \centering
        \begin{tikzpicture}[scale=0.65,%
	x={(1.0cm, 0.0cm)},
	z={(-0.20cm, -0.28cm)},
	y={(0.0cm, 1.0cm)},
    edgevisible/.style={draw=black!40, very thick},
    edgeinvisible/.style={draw=black!40, very thick, dashed},
    facevisible/.style={draw=none, fill=black!9, opacity=0.5},
    faceinvisible/.style={draw=none,},
    expoint/.style={fill=black},
    boxvisible/.style={thin, dashed},
    boxinvisible/.style={thin, dashed},
]

    \def\xishift{-8.0}

    \node(y0) at (0.0,8.25,\xishift) {};
    \node(y1) at (1.0,8.25,\xishift) {};
    \node(y2) at (3.869,7.827,\xishift) {};
    \node(y3) at (6.256,6.256,\xishift) {};
    \node(y4) at (7.827,3.869,\xishift) {};
    \node(y5) at (8.25,1.0,\xishift) {};
    \node(y6) at (8.25,0.0,\xishift) {};
    \node(y7) at (0.0,0.0,\xishift) {};

    \node(ysource) at (0,0,0) {};

    \node(box1) at (8.25,0,0) {};
    \node(box2) at (0,8.25,0) {};
    \node(box3) at (8.25,8.25,0) {};
    \node(box4) at (8.25,8.25,\xishift) {};

    \draw[faceinvisible] (y0.center) to (y1.center) to (y2.center) to (y3.center) to (y4.center) to (y5.center) to (y6.center) to (y7.center) to cycle;
    \draw[faceinvisible] (y0.center) to (y7.center) to (ysource.center) to cycle;
    \draw[faceinvisible] (y6.center) to (y7.center) to (ysource.center) to cycle;

    \draw[edgevisible] (y0.center) to (y7.center) ;
    \draw[edgevisible] (y6.center) to (y7.center) ;
    \draw[edgevisible] (ysource.center) to (y7.center);

	\draw[very thick,->] (-0.5,0,0) to node[pos=1.0, below] {$f_1(x)$} (9.75, 0,0);
	\draw[very thick,->] (0,-0.5,0) to node[pos=1.0, left] {$f_2(x)$} (0, 9.75,0);
	\draw[very thick,->] (0,0,1.0)  to (0,0,-10.25);

    \foreach \i [evaluate=\i as \j using int(\i+1)] in {0,1,2,3,4,5,5} {
        \draw[facevisible] (y\i.center) to (y\j.center) to (ysource.center) to cycle;
    }

    \foreach \i [evaluate=\i as \j using int(\i+1)] in {0,1,2,3,4,5} {
        \draw[edgevisible] (y\i.center) to (y\j.center);
    }
    \foreach \i in {0, 1,2,3,4,5,6} {
        \draw[edgevisible] (y\i.center) to (ysource.center);
    }
    \foreach \i in {1,2,3,4,5} {
        \draw[draw=none] (y\i.center) to node[pos=-0.08] {$y^{\i}$} (ysource.center);
    }
    \node at (0.5,-0.45,0) {$y^{0}$} (y3.center);

    \foreach \i in {1,2,3,4,5} {
        \fill[expoint] (y\i) circle (4pt);
    }
    \fill[expoint] (ysource) circle (4pt);


    \draw[boxvisible] (ysource.center) to (box1.center);
    \draw[boxvisible] (ysource.center) to (box2.center);
    \draw[boxvisible] (y6.center) to (box1.center);
    \draw[boxvisible] (y0.center) to (box2.center);
    \draw[boxvisible] (y6.center) to (box4.center);
    \draw[boxvisible] (y0.center) to (box4.center);
    \draw[boxvisible] (box1.center) to (box3.center);
    \draw[boxvisible] (box2.center) to (box3.center);
    \draw[boxvisible] (box4.center) to (box3.center);

	\draw[draw=none] (0,0,0.75)  to node[pos=0.95, above, xshift=3.5] {$-\xi$} (0,0,-10.25);

\end{tikzpicture}
    \end{minipage}\\[-16pt]
    \begin{minipage}{0.45\textwidth}
        \subcaption{$\Ykone$ in~$\mathbb{R}^2$\label{figure::maxNotApx::multicrit::bi}}
    \end{minipage}
    \hfill
    \begin{minipage}{0.45\textwidth}
        \subcaption{$\Yktwo$ in~$\mathbb{R}^3$\label{figure::maxNotApx::multicrit::tri}}
    \end{minipage}
    \caption{An example of the sets~$\Ykone$ and~$\Yktwo$ for~$\multiSingleBase$ and~$\multiBiBase$, respectively. The non-dominated extreme points $y^0,\ldots, y^N$ are ordered by increasing first components.\label{figure::maxNotApx::multicrit}}
\end{figure}

\medskip

Before considering the parametric variant of~\ref{problem::multiBiBase}, we first state two properties of the set~$\YEktwo$.
With some abuse of notation, for $y^i,y^j\in\YEktwo$ let
\[
\det(y^i,y^j)\coloneqq\det\begin{pmatrix}
    y_{1}^i & y_{1}^j \\ y_{2}^i & y_{2}^j
\end{pmatrix} = y_{1}^i\cdot y_{2}^j - y_{1}^j\cdot y_{2}^i,
\]
denote the determinant of the square matrix consisting of only the first two entries of~$y^i$ and~$y^j$.
Due to the ordering of $y^1,\ldots,y^N$, it holds that $\det(y^i,y^j)<0$ for $1\leq i < j$.

Furthermore, for every choice of~$i,j,\ell$ with $1\leq i < j < \ell \leq N$, it holds that
\begin{align}
\frac{\det(y^j,y^i)+\det(y^\ell,y^j) }{\det(y^\ell,y^i)} >\, & 1, \label{equation::maxNotApx::detineq1} \\
\frac{\det(y^\ell,y^i)-\det(y^j,y^i)}{\det(y^\ell,y^j)} <\, & 1, \label{equation::maxNotApx::detineq2} \\
\frac{\det(y^j,y^i)-\det(y^\ell,y^i)}{\det(y^\ell,y^j)} <\, & 1. \label{equation::maxNotApx::detineq3}
\end{align}
The first inequality is developed from the geometric properties of determinants, a formal proof is given in Lemma~\ref{lemma::maxNotApx::determinants} in the appendix.
The second and third inequality can directly be derived from the first.

\medskip

We now consider the $2$-parametric linear program~$\paraBiBase$ that we derive from~\ref{problem::multiBiBase}.
It consists of the parameter set~$\Lambda=\mathbb{R}_{\geq0}^2$, the feasible set~$\Xktwo$, and the parametric value function $F_\lambda(x)= \lambda_1 {c^1}^T x + \lambda_2 {c^2}^\T x - \xi$.
Clearly, it holds that $F^*(\lambda)\geq F_\lambda(x^0)=0$ for every $\lambda\in\Lambda$, and~$\paraBiBase$ can be encoded with length in $\bigO{(\poly(k))}$.

To give a lower bound on the cardinality of every $(1-\varepsilon)$-approximation set for~$\paraBiBase$, it is sufficient to focus on approximation sets that only contain solutions from~$\XEktwo$.
The reason is the following:
In an arbitrary solution set, every solution that is dominated in the $3$-objective problem~\ref{problem::multiBiBase} can be replaced by a solution dominating it, since the  parametric value function for the dominating solution is higher for every~$\lambda\in\Lambda$.
Furthermore, every non-dominated solution can be constructed as a convex combination of at most $3$~solutions from~$\XEktwo$.
Then, for every $\lambda\in\Lambda$, at least one of these three solutions must be as good or better for~$\paraBiBase(\lambda)$ as the original solution.
Therefore, there must be at least one $(1-\varepsilon)$-approximation set for~$\paraBiBase$ that only consists of solutions from~$\XEktwo$ and contains at most three times as many solutions as the $(1-\varepsilon)$-approximation set of minimum cardinality.
If we can show that all $(1-\varepsilon)$-approximation sets with solutions from~$\XEktwo$ have superpolynomial cardinality, it follows that the minimum cardinality set is also superpolynomial.

\medskip

We now show that every $(1-\varepsilon)$-approximation set must contain at least $N-2$~solutions from~$\XEktwo$ (in fact, it must contain all $N+1$~solutions, but we prove only the more straight-forward $N-2$ result).
As $N-2\in\Omega{\left(2^k\right)}$, this implies the aforementioned superpolynomial cardinality of the minimum cardinality $(1-\varepsilon)$-approximation set.

\begin{figure}
    \begin{minipage}[t]{0.47\textwidth}
        \centering
        \begin{tikzpicture}[scale=0.55,%
    edgevisible/.style={draw=black!40, very thick},
    edgesinvisible/.style={draw=none},
    rayvisible/.style={->},
    facevisible/.style={draw=none, fill=black!9, opacity=0.5},
]
    \coordinate (y0) at (0.000, 0.000);
    \coordinate (y1) at (0.000, 5.818);
    \coordinate (y2) at (0.000, 10.000);
    \coordinate (y3) at (0.917, 5.707);
    \coordinate (y4) at (2.607, 10.000);
    \coordinate (y5) at (3.344, 4.525);
    \coordinate (y6) at (4.525, 3.344);
    \coordinate (y7) at (5.707, 0.917);
    \coordinate (y8) at (5.818, 0.000);
    \coordinate (y9) at (7.091, 10.000);
    \coordinate (y10) at (10.000, 0.000);
    \coordinate (y11) at (10.000, 2.607);
    \coordinate (y12) at (10.000, 7.091);
    \coordinate (y13) at (10.000, 10.000);


    \draw[facevisible] (y0) to (y1) to (y3) to (y5) to (y6)to (y7) to (y8) to cycle;

    \draw[facevisible] (y1) to (y2) to (y4) to (y3) to cycle;
    \draw[facevisible] (y3) to (y4) to (y9) to (y5) to cycle;
    \draw[facevisible] (y5) to (y9) to (y13) to (y12) to (y6) to cycle;
    \draw[facevisible] (y6) to (y12) to (y11) to (y7) to cycle;
    \draw[facevisible] (y7) to (y11) to (y10) to (y8) to cycle;


    \draw[edgesinvisible] (y2) to (y4);
    \draw[edgesinvisible] (y4) to (y9);
    \draw[edgesinvisible] (y9) to (y13);
    \draw[edgesinvisible] (y13) to (y12);
    \draw[edgesinvisible] (y12) to (y11);
    \draw[edgesinvisible] (y11) to (y10);


    \draw[edgevisible] (y0) to (y1);
    \draw[edgevisible] (y1) to (y3);
    \draw[edgevisible] (y3) to (y5);
    \draw[edgevisible] (y5) to (y6);
    \draw[edgevisible] (y6) to (y7);
    \draw[edgevisible] (y7) to (y8);
    \draw[edgevisible] (y8) to (y0);


    \draw[edgevisible, rayvisible] (y1) to (y2);
    \draw[edgevisible, rayvisible] (y3) to (y4);
    \draw[edgevisible, rayvisible] (y5) to (y9);
    \draw[edgevisible, rayvisible] (y6) to (y12);
    \draw[edgevisible, rayvisible] (y7) to (y11);
    \draw[edgevisible, rayvisible] (y8) to (y10);

    \node at (2.0,2.0) {$\Lambda(x^0)$};
    \node at (0.95,8.5) {$\Lambda(x^1)$};
    \node at (3.3,7.7) {$\Lambda(x^2)$};
    \node at (5.7,5.7) {$\Lambda(x^3)$};
    \node at (7.3,3.2) {$\Lambda(x^4)$};
    \node at (8.5,0.7) {$\Lambda(x^5)$};

    \node at (5.0,-0.75) {$\lambda_1$};
    \node at (-0.75,5.0) {$\lambda_2$};

\end{tikzpicture}
    \end{minipage}
    \hfill
    \begin{minipage}[t]{0.47\textwidth}
        \centering
        \begin{tikzpicture}[scale=0.4,
    zeroedge/.style={draw=black!40, very thick},
    soledge/.style={draw=black!40, very thick},
    deltaedge/.style={draw=black!40, very thick, dashed},
    rayvisible/.style={->},
    facevisible/.style={draw=none, fill=black!9, opacity=0.5},
    facedelta/.style={draw=none, fill=black!0.4, opacity=0.5},
    ]

    \coordinate (y1) at (0.000, 6.000);
    \coordinate (y2) at (6.000, 0.000);
    \coordinate (y3) at (4.9, 4.9);
    \draw[draw=none] ($(y3)!1.8!(y1)$) to node[pos=0.0](y4){} (y1);
    \draw[draw=none] ($(y3)!1.8!(y2)$) to node[pos=0.0](y5){} (y2);
    \coordinate (y6) at (3.0, 11.0);
    \coordinate (y7) at (11.0, 3.0);
    \coordinate (y8) at (11.0, 11.0);

    \draw[facevisible] (y1) to (y3) to (y2) to (y7) to (y8) to (y6) to cycle;
    \draw[facevisible] ($(y3)!1.8!(y1)$) to (y1) to (y6) to (-3.92, 11.00) to cycle;
    \draw[facevisible] ($(y3)!1.8!(y2)$) to (y2) to (y7) to (11.00,-3.92) to cycle;
    \draw[facevisible] (-3.92,-3.92) to (-3.92,6.98) to (y1) to (y2) to (6.98,-3.92) to cycle;
    \draw[facedelta] (y1) to (y2) to (y3) to cycle;

    \draw[zeroedge] (y1) to (y2);
    \draw[soledge] ($(y3)!1.8!(y1)$) to node[pos=0.0](y4){} (y1);
    \draw[soledge] ($(y3)!1.8!(y2)$) to node[pos=0.0](y5){} (y2);
    \draw[deltaedge] (y3) to (y1);
    \draw[deltaedge] (y3) to (y2);

    \coordinate (y6) at (3.0, 11.0);
    \coordinate (y7) at (11.0, 3.0);
    \draw[soledge, rayvisible] (y1) to (y6);
    \draw[soledge, rayvisible] (y2) to (y7);

    \node at (3.8,3.8) {$\Delta(3)$};
    \node at (0.0,0.0) {$\Lambda(x^0)$};
    \node at (7.5,7.5) {$\Lambda(x^3)$};
    \node at (-1.0,9.0) {$\Lambda(x^{2})$};
    \node at (9.0,-1.0) {$\Lambda(x^{4})$};

    \draw[draw=none] ($(y3)!1.7!(y1)$) to node[pos=0.45,sloped, yshift=-6]{\scriptsize\thinmuskip=0mu\medmuskip=0mu\thickmuskip=2mu$F_\lambda(x^{2})= 0$} (y1);
    \draw[draw=none] (y2) to node[pos=0.45,sloped,yshift=-6]{\scriptsize\thinmuskip=0mu\medmuskip=0mu\thickmuskip=2mu$F_\lambda(x^{4})=0$} ($(y3)!1.8!(y2)$);
    \draw[draw=none] (y1) to node[pos=0.5,sloped,yshift=-6]{\scriptsize\thinmuskip=0mu\medmuskip=0mu\thickmuskip=2mu$F_\lambda(x^{3})=0$} (y2);

    \filldraw[draw=black!40,fill=black!0.4] (y3) circle (7pt);

    \node[xshift=10pt,yshift=5pt] at (y3) {$\lambda^i$};

\end{tikzpicture}
    \end{minipage}\\[-0pt]
    \begin{minipage}[t]{0.46\textwidth}
        \subcaption{The optimality regions for the example of~$\paraBiBase$ from Figure~\ref{figure::maxNotApx::multicrit::tri}\label{figure::maxNotApx::paraspace::struct}}
    \end{minipage}
    \hfill
    \begin{minipage}[t]{0.46\textwidth}
        \subcaption{An example of the set~$\nonApxArea(3)$ and the construction of~$\lambda^3$.\label{figure::maxNotApx::paraspace::delta}}
    \end{minipage}
    \caption{A geometric representation of the optimality regions for~$\paraBiBase$ and for~$\Delta(3)$ from the proof of Lemma~\ref{lemma::maxNotApx::DeltaVertex}.
    The optimality region~$\Lambda(x)$ for a solution~$x$ is the set $\{\lambda\in\Lambda: F_\lambda(x)=F^*(\lambda)\}$.
    Optimality regions are polyhedra and can only intersect at their boundaries (see~\cite{przybylski.gandibleux.ea2009RecursiveAlgorithmFinding}).
    The set~$\Delta(i)$ is a subset of the set~$\Lambda(x^i)$, and the point~$\lambda^i$ can be seen as the intersection of the lines containing the edges between~$\Lambda(x^0)$ and~$\Lambda(x^{i-1})$, and between~$\Lambda(x^0)$ and~$\Lambda(x^{i+1})$.%
    \label{figure::maxNotApx::paraspace}}

\end{figure}
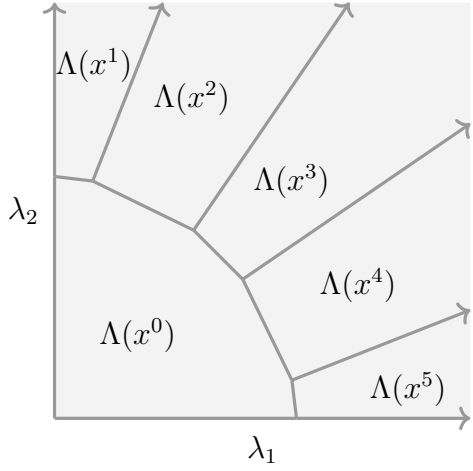
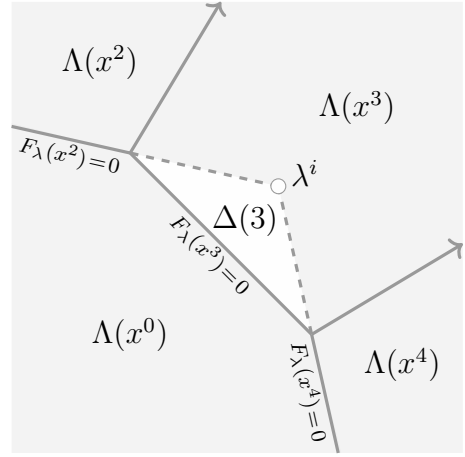

\begin{definition}\label{definition::maxNotApx::Delta}
    For $i\in\{1,\ldots,N\}$, we define
    \[
    \nonApxArea(i)\coloneqq
    \left\{\lambda\in\Lambda: \quad
    \begin{matrix}
        F_\lambda(x^i)> F_\lambda(x^0), \\
         F_\lambda(x^j)\leq F_\lambda(x^0)\phantom{,} & \forall\  j\in\{1,\ldots,N\}\setminus \{i\}
    \end{matrix}
    \right\}.
    \]
\end{definition}
The set $\nonApxArea(i)$ contains all parameter vectors~$\lambda\in\Lambda$ where~$F_\lambda(x^i)$ is strictly better than~$F_\lambda(x^0)$, and where, for each $x^j\in \XEktwo\setminus\{x^i\}$, the value of~$F_\lambda(x^j)$ is not better than of~$F_\lambda(x^0)$.
By construction of $x^0$, we have $F_\lambda(x^0)=0$ for all~$\lambda\in\Lambda$.
Accordingly, for~$\lambda\in\nonApxArea(i)$, the only solution from~$\XEktwo$ that can $(1-\varepsilon)$-approximate~$\paraBiBase(\lambda)$ is~$x^i$, since  $F_\lambda(x^j)\leq0$ for every other solution~$j\neq i$.
We now show $\nonApxArea(i)\neq \emptyset$ for every $1<i<N$.
For this, we pick the parameter vector~$\lambda^i$ as the solution of the system of equalities given by $F_{\lambda^i}(x^{i-1})=0$ and $F_{\lambda^i}(x^{i+1})=0$, and show that $\lambda^i\in\Delta(i)$.
A graphical interpretation is provided in Figure~\ref{figure::maxNotApx::paraspace}.

\begin{lemma}\label{lemma::maxNotApx::DeltaVertex}
    For $1<i<N$, let $\lambda^i\in\mathbb{R}^2$ be defined by
    \begin{equation}\label{equation::maxNotApx::lambda2solved}
    \begin{aligned}
        \lambda^i_1 &\coloneqq\frac{y_{2}^{i-1}-y_{2}^{i+1}}{\det(y^{i+1},y^{i-1})}, \\
        \lambda^i_2 &\coloneqq\frac{y_{1}^{i+1}-y_{1}^{i-1}}{\det(y^{i+1},y^{i-1})}.
    \end{aligned}
    \end{equation}
    Then, $\lambda^i\in\Delta(i)$.
\end{lemma}
\begin{proof}
    Let~$i$ be fixed such that $1<i<N$.
    The first condition for $\lambda^i\in\Delta(i)$ is that $\lambda^i\in\Lambda=\mathbb{R}_{\geq0}^2$.
    Recall that, by construction, $y_2^{i-1}>y_2^{i+1}$ and $y_1^{i+1}>y_1^{i-1}$, and, since $i+1>i-1$, we have $\det(y^{i+1},y^{i-1})>0$.
    Thus, $\lambda^i\in\mathbb{R}_{\geq0}^2$.

    For the second condition, we need to show that the strict inequality $F_{\lambda^i}(x^i)>F_{\lambda^i}(x^0)$  is satisfied.
    By construction, $F_{\lambda}(x^0)=0$ for all~$\lambda\in\Lambda$.
    Thus,
    \begin{equation}\label{equation::maxNotApx::FiEqualToZero}
    \begin{alignedat}{2}
        F_{\lambda^i}(x^i)&=\lambda_1^i y_{1}^i+\lambda_2^i y_{2}^i-1 \\
        \overset{(\ref{equation::maxNotApx::lambda2solved})}&{=}
            \frac{y_{2}^{i-1}-y_{2}^{i+1}}{\det(y^{i+1},y^{i-1})}y_{1}^i
            + \frac{y_{1}^{i+1}-y_{1}^{i-1}}{\det(y^{i+1},y^{i-1})}y_{2}^i
            -1 \\
        &= \frac{y_1^i y_2^{i-1} - y_2^{i}y_1^{i-1} + y_1^{i+1}y_2^{i}-y_2^{i+1} y_1^{i}}{\det(y^{i+1}, y^{i-1})} -1 \\
        &= \frac{\det(y^i, y^{i-1})+\det(y^{i+1}, y^{i})}{\det(y^{i+1}, y^{i-1})} -1 \\
        \overset{(\ref{equation::maxNotApx::detineq1})}&> 1 - 1 = 0.
    \end{alignedat}
    \end{equation}
    For the last condition, we need to show that $F_{\lambda^i}(x^j)\leq F_{\lambda^i}(x^0) = 0$ for all $j\in\{1,\ldots,N\}\setminus\{i\}$.
    First, let $j<i$.
    Similar to Equation~\eqref{equation::maxNotApx::FiEqualToZero}, for $j=i-1$, we obtain
    \begin{equation*}
        F_{\lambda^i}(x^{j})
        = \frac{\det(y^{i+1}, y^{j})-\det(y^{i-1},y^{j})}{\det(y^{i+1}, y^{i-1})} -1
         = \frac{\det(y^{i+1}, y^{i-1})}{\det(y^{i+1}, y^{i-1})} -1 = 0,
    \end{equation*}
    and for $j<i-1$, we obtain
    \begin{equation*}
        F_{\lambda^i}(x^j)
        = \frac{\det(y^{i+1}, y^{j})-\det(y^{i-1},y^j)}{\det(y^{i+1}, y^{i-1})} -1
         \overset{(\ref{equation::maxNotApx::detineq2})}{<} 1-1=0.
    \end{equation*}
    With a symmetric argument, for $j>i$, we obtain
    \begin{equation*}
        F_\lambda(x^j)
        = \frac{\det(y^{j},y^{i-1})-\det(y^{j},y^{i+1})}{\det(y^{i+1}, y^{i-1})} -1
        \overset{(\ref{equation::maxNotApx::detineq3})}{\leq} 1-1=0.
    \end{equation*}
\end{proof}

As a consequence of Lemma~\ref{lemma::maxNotApx::DeltaVertex}, for every $1<i<N$, the solution~$x^i$ is the only solution from~$\XEktwo$ that can $(1-\varepsilon)$-approximate the problem~$\paraBiBase{\left(\lambda^i\right)}$.
Therefore, a $(1-\varepsilon)$-approximation set consisting of solutions from~$\XEktwo$ must have a cardinality of at least $N-2\in\Omega(2^k)$.\footnote{
In fact, $x^0$, $x^1$, and~$x^{N}$ must also be contained in such an approximation set.
We omit a proof since the statement from Lemma~\ref{lemma::maxNotApx::DeltaVertex} is already sufficient to obtain the central result of this section.
}

\medskip

The construction of~$\paraBiBase$ proves Theorem~\ref{theorem::maxNotApx} under the assumption that $F^*(\lambda)\geq0$ for all~$\lambda\in\Lambda$.
We now show that, for every constant~$K>0$, we can easily modify~$\paraBiBase$ so that $F^*(\lambda)\geq K$ is ensured, the encoding length is still in~$\bigO(\poly(k))$, and every $(1-\varepsilon)$-approximation set has a cardinality in~$\Omega(2^k)$.
Without loss of generality, we can assume that~$K\in\mathbb{Q}$ (simply ``rounding'' up to a bigger rational number is always possible).

We modify~$\paraBiBase$ in two steps.
First, we add an absolute offset.
We define the parametric value function $F^u_\lambda(x)\coloneqq F_\lambda(x)+u$
and choose~$u$ as
\[
u \coloneqq (1-\varepsilon)\cdot\min_{1<i<N}F_{\lambda^i}(x^i).
\]
By Lemma~\ref{lemma::maxNotApx::DeltaVertex}, we obtain that~$u>0$.
Also, for constant~$\varepsilon$, we have $\langle u\rangle\in\bigO(\poly (k))$:
For every $1<i<N$, $\lambda_i$ is defined as the intersection of the two hyperplanes defined by the equalities $F_\lambda(x^{i-1})=0$ and $F_\lambda(x^{i+1})=0$ and, therefore, $\langle \lambda^i\rangle\in\bigO(\poly (k))$.
Thus, if~$F^u_\lambda$ is used as the parametric value function for $\paraBiBase$, the encoding length is still in $\bigO(\poly(k))$, and $({F^u})^*(\lambda)\geq F_\lambda^u(x^0) = u > 0$ for all $\lambda\in\Lambda$ holds.
Furthermore, for every $1<i<N$, it holds that
\begin{alignat*}{2}
    (1-\varepsilon)F^{u}_{\lambda^i}(x^i) & = (1-\varepsilon)F_{\lambda^i}(x^i)+(1-\varepsilon)u \\
    & \geq (1-\varepsilon)\left(\min_{1<j<N} F_{\lambda^j}(x^j)\right)+(1-\varepsilon)u \\
    & = u + (1-\varepsilon) u \\
    & > u \\
    & = F_{\lambda^i}(x^0)
\end{alignat*}

Again, $F_{\lambda^i}(x^0)>F_{\lambda^i}(x^j)$ for every $1<j<N$ and, therefore, $x^i$ is the only solution in~$\XEktwo$ that can $(1-\varepsilon)$-approximate~$\Pi(\lambda^i)$.
Every $(1-\varepsilon)$-approximation set needs at least~$\Omega(2^k)$ solutions.

In the second step, we bound the optimal objective function values from below by $K>0$ for every~$\lambda\in\Lambda$.
For this, we simply scale the parametric value function by the factor~$\nicefrac{K}{u}$, i.e., we construct the parametric value function $F_\lambda^K(x)=\nicefrac{K}{u} \cdot F_\lambda^u(x)$.%
\footnote{In the multi-objective setting, this can be seen as replacing the objectives $({c^1}^\T x,{c^2}^\T x,-\xi)$ of~$\multiBiBase$ by $\left(\nicefrac{K}{u}\cdot({c^1}^\T x),\nicefrac{K}{u}\cdot({c^2}^\T x),\nicefrac{K}{u}\cdot(u-\xi)\right)$.
}
Clearly, $F_\lambda^K(x^0)= K$ for every $\lambda\in\Lambda$.
Scaling by a positive factor does not, however, alter the optimal or $(1-\varepsilon)$-approximate solutions of~$\Pi(\lambda)$ for a given~$\lambda \in \Lambda$.
Therefore, the cardinality of every $(1-\varepsilon)$-approximation set is again in~$\Omega(2^k)$.
Since $\langle u\rangle\in \bigO(\poly(k))$ and~$K$ is assumed to be constant, the encoding length of this modified~$\paraBiBase$ is still in~$\bigO(\poly(k))$.

\section{Conclusion}

We presented a new algorithm for (linear-multi-)parametric optimization.
It can be used to solve parametric optimization problem instances exactly by determining an optimal solution set, or approximately by determining an approximation set.
As long as an appropriate oracle for the non-parametric variant of a problem is available, the former case runs in output-polynomial time, and the latter in polynomial time.
The assumptions needed to apply our algorithm are less strict than the assumptions made by previous parametric approximation algorithms.
In addition, we showed that a relaxation of these assumptions for parametric maximization problems leads to a class of problems that can never be approximated, regardless of the existence of a polynomial-time oracle.

Our results can also be applied to multi-objective optimization for computing convex approximation sets~\cite{diakonikolas2011ApproximationMultiobjective}.
As shown in~\cite{hamel.lohne.ea2014BensonTypeAlgorithms}, the so-called \emph{domination cone} of a multi-objective problem instance directly corresponds to the parameter set of a related parametric optimization problem instance.
Without proof, we remark that our algorithm can  be applied to this parametric optimization problem instance to compute a convex approximation set for the original multi-objective optimization problem instance.
To the best of our knowledge, this makes our algorithm the first polynomial-time approximation algorithm for multi-objective optimization problems with general polyhedral domination cones.

Although parametric approximation algorithms ensure polynomial running times, many of these algorithms find approximation sets of much higher cardinality than actually necessary.
This is demonstrated by computational studies in~\cite{helfrich.ruzika.ea2025EfficientlyConstructing,nemesch.ruzika.ea2025PolynomialTimeInnerApproximation}.
Thus, a goal in future research could be to investigate parametric approximation algorithms that compute approximation sets with a cardinality bounded in some form by the cardinality of the smallest possible approximation set.
Such an algorithm exists for the $1$-parametric case~\cite{bazgan.herzel.ea2022ApproximationAlgorithmGeneral}, but for the general multi-parametric case only some lower bounds are known~\cite{helfrich.herzel.ea2022ApproximationAlgorithmGeneral}, but no algorithms have been developed.

\FloatBarrier

\AtNextBibliography{\small}
\printbibliography[title={References},heading={bibintoc}]

@article{bazgan.herzel.ea2022ApproximationAlgorithmGeneral,
  title = {An Approximation Algorithm for a General Class of Parametric Optimization Problems},
  author = {Bazgan, Cristina and Herzel, Arne and Ruzika, Stefan and Thielen, Clemens and Vanderpooten, Daniel},
  year = {2022},
  journal = {Journal of Combinatorial Optimization},
  volume = {43},
  number = {5},
  pages = {1328--1358},
  doi = {10.1007/s10878-020-00646-5}
}

@inproceedings{below.loera.ea2000FindingMinimalTriangulations,
    author = {Below, A. and De Loera, J.A. and Richter-Gebert, J.},
    booktitle = {Proceedings of the 11th Annual ACM-SIAM Symposium on Discrete Algorithms ({SODA})},
    title = {Finding minimal triangulations of convex 3-polytopes is {NP}-hard},
    year = {2000},
    pages = {65--66},
    url = {https://dl.acm.org/doi/10.5555/338219.338235}
}

@inproceedings{bokler.mutzel2015OutputSensitiveAlgorithmsEnumerating,
  title = {Output-Sensitive Algorithms for Enumerating the Extreme Nondominated Points of Multiobjective Combinatorial Optimization Problems},
  booktitle = {Proceedings of the 23rd European Symposium on Algorithms ({ESA})},
  author = {Bökler, Fritz and Mutzel, Petra},
  year = {2015},
  pages = {288--299},
  doi = {10.1007/978-3-662-48350-3_25}
}

@phdthesis{bokler2018OutputsensitiveComplexityMultiobjective,
  title = {Output-Sensitive Complexity of Multiobjective Combinatorial Optimization Problems with an Application to the Multiobjective Shortest Path Problem},
  author = {Bökler, Fritz},
  year = {2018},
  school = {Technische Universität Dortmund},
  doi = {10.17877/DE290R-19130}
}

@inproceedings{bokler.nemesch.ea2023PaMILOSolverMultiobjective,
  author = {Bökler, Fritz and Nemesch, Levin and Wagner, Mirko H.},
  booktitle = {Operations Research Proceedings 2022},
  year = {2023},
  doi = {10.1007/978-3-031-24907-5_20},
  pages = {163--170},
  title = {{PaMILO}: A Solver for Multi-objective Mixed Integer Linear Optimization and Beyond}
}

@inproceedings{borndorfer.schenker.ea2016PolySCIP,
  author = {Borndörfer, Ralf and Schenker, Sebastian and Skutella, Martin and Strunk, Timo},
  booktitle = {Proceedings of the 5th International Congres on Mathematical Software ({ICMS})},
  year = {2016},
  doi = {10.1007/978-3-319-42432-3_32},
  pages = {259--264},
  title = {{PolySCIP}}
}

@article{carstensen1983ComplexityParametric,
  author = {Carstensen, P. J.},
  year = {1983},
  doi = {10.1007/BF02591893},
  journal = {Mathematical Programming},
  pages = {64--75},
  title = {Complexity of some parametric integer and network programming problems},
  volume = {26},
}

@article{chazelle1993OptimalConvexHull,
  title = {An Optimal Convex Hull Algorithm in Any Fixed Dimension},
  author = {Chazelle, Bernard},
  year = {1993},
  journal = {Discrete \& Computational Geometry},
  volume = {10},
  pages = {377--409},
  doi = {10.1007/BF02573985}
}

@article{csirmaz2021InnerApproximationAlgorithm,
  author = {Csirmaz, Laszlo},
  publisher = {Taylor \& Francis},
  year = {2021},
  doi = {10.1080/02331934.2020.1737692},
  journal = {Optimization},
  number = {7},
  pages = {1487--1511},
  title = {Inner Approximation Algorithm for Solving Linear Multiobjective Optimization Problems},
  volume = {70},
}

@article{daskalakis.diakonikola2016HowGood,
  author = {Daskalakis, C. and Diakonikolas, I. and Yannakakis, M.},
  year = {2016},
  doi = {10.1137/13093875X},
  journal = {SIAM Journal on Computing},
  number = {3},
  pages = {811--858},
  title = {How Good is the Chord Algorithm?},
  volume = {45},
}

@book{deloera.rambau.ea2010TriangulationsStructuresAlgorithms,
  title = {Triangulations: Structures for Algorithms and Applications},
  author = {De Loera, Jesus A. and Rambau, Jorg and Santos, Francisco},
  year = {2010},
  edition = {1},
  publisher = {Springer},
  doi = {10.1007/978-3-642-12971-1}
}

@inproceedings{diakonikolas.yannakakis2008SuccinctApproximateConvex,
author = {Diakonikolas, Ilias and Yannakakis, Mihalis},
title = {Succinct approximate convex {Pareto} curves},
year = {2008},
booktitle = {Proceedings of the Nineteenth Annual {ACM-SIAM} Symposium on Discrete Algorithms ({SODA})},
pages = {74--83},
}

@phdthesis{diakonikolas2011ApproximationMultiobjective,
  author = {Diakonikolas, I.},
  school = {Graduate School of Arts and Sciences, Columbia University},
  year = {2011},
  doi = {10.7916/D80K2GR9},
  title = {Approximation of Multiobjective Optimization Problems},
}

@book{ehrgott2005multicriteriaOptimization,
  author = {Ehrgott, Matthias},
  publisher = {Springer Berlin, Heidelberg},
  year = {2005},
  doi = {10.1007/3-540-27659-9},
  edition = {2},
  title = {Multicriteria Optimization},
}

@article{ehrgott.lohne.ea2012DualVariantBenson,
  author = {Ehrgott, Matthias and Löhne, Andreas and Shao, Lizhen},
  year = {2012},
  doi = {10.1007/s10898-011-9709-y},
  journal = {Journal of Global Optimization},
  number = {4},
  pages = {757--778},
  title = {A Dual Variant of {{Benson}}'s ``Outer Approximation Algorithm'' for Multiple Objective Linear Programming},
  volume = {52},
}

@article{ehrgott.shao.ea2011ApproximationAlgorithmConvex,
  title = {An Approximation Algorithm for Convex Multi-Objective Programming Problems},
  author = {Ehrgott, Matthias and Shao, Lizhen and Sch\"obel, Anita},
  year = {2011},
  journal = {Journal of Global Optimization},
  volume = {50},
  number = {3},
  pages = {397--416},
  doi = {10.1007/s10898-010-9588-7},
}

@article{eisner.severance1976MathematicalTechniques,
  author = {Eisner, M. J. and Severance, D. G.},
  year = {1976},
  doi = {10.1145/321978.321982},
  journaltitle = {Journal of the ACM},
  number = {4},
  pages = {619--635},
  title = {Mathematical Techniques for Efficient Record Segmentation in Large Shared Databases},
  volume = {23},
}

@article{fernandez-baca.seppalainen.ea2004ParametricMultipleSequence,
  author = {Fernández-Baca, David and Seppäläinen, Timo and Slutzki, Giora},
  year = {2004},
  doi = {10.1016/S1570-8667(03)00078-9},
  journal = {Journal of Discrete Algorithms},
  number = {2},
  pages = {271--287},
  title = {Parametric Multiple Sequence Alignment and Phylogeny Construction},
  volume = {2},
}

@article{fernandez-baca.srinivasan1991ConstructingMinimizationDiagrama,
  author = {Fernández-Baca, David and Srinivasan, S.},
  year = {1991},
  doi = {10.1016/0167-6377(91)90092-4},
  journal = {Operations Research Letters},
  number = {2},
  pages = {87--93},
  title = {Constructing the Minimization Diagram of a Two-Parameter Problem},
  volume = {10},
}

@article{giudici.halffmann2017ApproximationSchemes,
  author = {Giudici, A. and Halffmann, P. and Ruzika, S. and Thielen, C.},
  year = {2017},
  doi = {10.1016/j.ipl.2016.12.003},
  journal = {Information Processing Letters},
  pages = {11--15},
  title = {Approximation schemes for the parametric knapsack problem},
  volume = {120},
}

@inproceedings{glasser.reiwiesner2010ApproximabilityHardness,
  title = {Approximability and Hardness in Multi-objective Optimization},
  booktitle = {Proceedings of the 6th Conference on Computability in Europe ({CiE})},
  author = {Gla{\ss}er, C. and Reitwie{\ss}ner, C. and Schmitz, H. and Witek, M.},
  year = {2010},
  pages = {180--189},
  doi = {10.1007/978-3-642-13962-8_20},
}

@incollection{grotschel.lovasz.ea1988RationalPolyhedra,
  title = {Rational Polyhedra},
  booktitle = {Geometric Algorithms and Combinatorial Optimization},
  author = {Grötschel, Martin and Lovász, László and Schrijver, Alexander},
  date = {1988},
  pages = {157--196},
  publisher = {Springer},
  doi = {10.1007/978-3-642-97881-4_7}
}

@article{halman.holzhauser.ea2018FPTAKnapsackParaWeights,
  title = {An {FPTAS} for the knapsack problem with parametric weights},
  author = {Nir Halman and Michael Holzhauser and Sven O. Krumke},
  journal = {Operations Research Letters},
  volume = {46},
  number = {5},
  pages = {487--491},
  year = {2018},
  doi = {10.1016/j.orl.2018.07.005},
}

@article{hamel.lohne.ea2014BensonTypeAlgorithms,
  author = {Hamel, Andreas H. and Löhne, Andreas and Rudloff, Birgit},
  year = {2014},
  doi = {10.1007/s10898-013-0098-2},
  journal = {Journal of Global Optimization},
  langid = {english},
  number = {4},
  pages = {811--836},
  title = {Benson Type Algorithms for Linear Vector Optimization and Applications},
  volume = {59},
}

@article{helfrich.herzel.ea2022ApproximationAlgorithmGeneral,
  title = {An Approximation Algorithm for a General Class of Multi-Parametric Optimization Problems},
  author = {Helfrich, Stephan and Herzel, Arne and Ruzika, Stefan and Thielen, Clemens},
  year = {2022},
  journal = {Journal of Combinatorial Optimization},
  volume = {44},
  pages = {1459--1494},
  doi = {10.1007/s10878-022-00902-w}
}

@article{helfrich.ruzika.ea2025EfficientlyConstructing,
  author = {Helfrich, Stephan and Ruzika, Stefan and Thielen, Clemens},
  title = {Efficiently Constructing Convex Approximation Sets in Multiobjective Optimization Problems},
  journal = {INFORMS Journal on Computing},
  year = {2025},
  volume = {37},
  number = {5},
  pages = {1223--1241},
  doi = {10.1287/ijoc.2023.0220}
}

@article{herzel.ruzika.ea2021ApproximationMethodsMultiobjective,
  title = {Approximation Methods for Multiobjective Optimization Problems: A Survey},
  author = {Herzel, Arne and Ruzika, Stefan and Thielen, Clemens},
  year = {2021},
  journal = {INFORMS Journal on Computing},
  volume = {33},
  number = {4},
  pages = {1284--1299},
  doi = {10.1287/ijoc.2020.1028}
}

@article{heyde.lohne2008GeometricDualityMultiple,
  author = {Heyde, Frank and Löhne, Andreas},
  year = {2008},
  doi = {10.1137/060674831},
  journal = {SIAM Journal on Optimization},
  number = {2},
  pages = {836--845},
  title = {Geometric Duality in Multiple Objective Linear Programming},
  volume = {19},
}

@article{holzhauser.krumke2017AnFPTAS,
  author = {Holzhauser, M. and Krumke, S. O.},
  year = {2017},
  doi = {10.1016/j.ipl.2017.06.006},
  journal = {Information Processing Letters},
  pages = {43--47},
  title = {An {FPTAS} for the parametric knapsack problem},
  volume = {126},
}

@inbook{kaibel2011BasicPolyhedralTheory,
  author = {Kaibel, Volker},
  publisher = {John Wiley \& Sons, Ltd},
  title = {Basic Polyhedral Theory},
  booktitle = {Wiley Encyclopedia of Operations Research and Management Science},
  doi = {https://doi.org/10.1002/9780470400531.eorms0091},
  year = {2011}
}

@article{isermann1977EnumerationSetAll,
  author = {Heinz Isermann},
  title = {The Enumeration of the Set of All Efficient Solutions for a Linear Multiple Objective Program},
  journal = {Journal of the Operational Research Society},
  volume = {28},
  number = {3},
  pages = {711--725},
  year = {1977},
  doi = {10.1057/jors.1977.147},
}

@article{johnson.yannakakis.ea1988GeneratingAllMaximal,
  title = {On Generating All Maximal Independent Sets},
  author = {Johnson, David S. and Yannakakis, Mihalis and Papadimitriou, Christos H.},
  year = {1988},
  journaltitle = {Information Processing Letters},
  volume = {27},
  number = {3},
  pages = {119--123},
  doi = {10.1016/0020-0190(88)90065-8},
}

@article{katoh.ibaraki1987ParametricCharacterizationEapproximation,
  author = {Katoh, Naoki and Ibaraki, Toshihide},
  year = {1987},
  doi = {10.1016/0166-218X(87)90006-0},
  journal = {Discrete Applied Mathematics},
  number = {1},
  pages = {39--66},
  title = {A Parametric Characterization and an {$\varepsilon$}-Approximation Scheme for the Minimization of a Quasiconcave Program},
  volume = {17},
}

@incollection{lee.santos2017SubdivisionsTriangulations,
  author = {Lee, Carl W. and Santos, Francisco},
  booktitle = {Handbook of Discrete and Computational Geometry},
  publisher = {CRC Press},
  title = {Subdivisions and Triangulations of Polytopes},
  year = {2017},
  edition = {3},
  chapter = {16},
  pages = {415--447},
  doi = {10.1201/9781315119601},
}

@article{liu.ehrgott2018PrimalDualAlgorithms,
  title = {Primal and Dual Algorithms for Optimization over the Efficient Set},
  author = {Liu, Zhengliang and Ehrgott, Matthias},
  year = {2018},
  journal = {Optimization},
  volume = {67},
  number = {10},
  pages = {1661--1686},
  doi = {10.1080/02331934.2018.1484922},
}

@article{lohne.rudloff.ea2014PrimalDualApproximation,
  title = {Primal and Dual Approximation Algorithms for Convex Vector Optimization Problems},
  author = {Löhne, Andreas and Rudloff, Birgit and Ulus, Firdevs},
  year = {2014},
  journaltitle = {Journal of Global Optimization},
  volume = {60},
  pages = {713--736},
  doi = {10.1007/s10898-013-0136-0}
}

@article{lohne.weissing2017VectorLinearProgram,
  author = {Löhne, Andreas and Weißing, Benjamin},
  year = {2017},
  doi = {10.1016/j.ejor.2016.02.039},
  journal = {European Journal of Operational Research},
  number = {3},
  pages = {807--813},
  title = {The Vector Linear Program Solver {{{\emph{Bensolve}}}} -- Notes on Theoretical Background},
  volume = {260},
}

@article{murty1980ComputationalComplexity,
  author = {Murty, K. G.},
  year = {1980},
  doi = {10.1007/BF01581642},
  journal = {Mathematical Programming},
  number = {1},
  pages = {213--219},
  title = {Computational Complexity of Parametric Linear Programming},
  volume = {19},
}

@article{nemesch.ruzika.ea2025PolynomialTimeInnerApproximation,
  title = {A Polynomial-Time Inner Approximation Algorithm for Multiobjective and Parametric Optimization},
  author = {Nemesch, Levin and Ruzika, Stefan and Thielen, Clemens and Wittmann, Alina},
  year = {2026},
  journal = {INFORMS Journal on Computing},
  doi = {10.1287/ijoc.2025.1308}, 
}

@article{nemesch.ruzika.ea2025SurveyOfExact,
  author={Nemesch, Levin and Ruzika, Stefan and Thielen, Clemens and Wittmann, Alina},
  title={A survey of exact and approximation algorithms for linear-parametric optimization problems},
  journal={Journal of Global Optimization},
  year={2025},
  volume = {93},
  pages = {299--333},
  doi={10.1007/s10898-025-01512-6},
}

@inproceedings{papadimitriou.yannakakis2000ApproximabilityTradeoffsOptimal,
  title = {On the Approximability of Trade-Offs and Optimal Access of Web Sources},
  booktitle = {Proceedings of the 41st Annual Symposium on Foundations of Computer Science (FOCS)},
  author = {Papadimitriou, C.H. and Yannakakis, M.},
  year = {2000},
  pages = {86--92},
  doi = {10.1109/SFCS.2000.892068}
}

@article{przybylski.gandibleux.ea2009RecursiveAlgorithmFinding,
  author = {Przybylski, Anthony and Gandibleux, Xavier and Ehrgott, Matthias},
  publisher = {INFORMS},
  year = {2009},
  doi = {10.1287/ijoc.1090.0342},
  journal = {INFORMS Journal on Computing},
  number = {3},
  pages = {371--386},
  title = {A Recursive Algorithm for Finding All Nondominated Extreme Points in the Outcome Set of a Multiobjective Integer Programme},
  volume = {22},
}

@article{seidel1995UpperBoundTheorem,
  title = {The Upper Bound Theorem for Polytopes: An Easy Proof of Its Asymptotic Version},
  author = {Seidel, Raimund},
  date = {1995},
  journaltitle = {Computational Geometry},
  volume = {5},
  number = {2},
  pages = {115--116},
  doi = {10.1016/0925-7721(95)00013-Y}
}

@article{ruhe1988ComplexityResults,
  author = {Ruhe, G.},
  year = {1988},
  doi = {10.1007/BF01920568},
  journaltitle = {Zeitschrift für Operations Research},
  pages = {9--27},
  title = {Complexity results for multicriterial and parametric network flows using a pathological graph of {Zadeh}},
  volume = {32},
}

@book{schrijver1998TheoryLinearInteger,
  author = {Schrijver, A.},
  title = {Theory of Linear and Integer Programming},
  year = {1998},
  publisher = {Wiley},
  isbn = {978-0-471-98232-6},
}

@book{vazirani2001ApproximationAlgorithms,
  title = {Approximation Algorithms},
  author = {Vazirani, Vijay V.},
  year = {2003},
  publisher = {Springer},
  edition = {first},
  doi = {10.1007/978-3-662-04565-7},
}

@book{valenza1993LinearAlgebra,
  title = {Linear Algebra},
  author = {Valenza, Robert J.},
  year = {1993},
  publisher = {Springer},
  doi = {10.1007/978-1-4612-0901-0},
}

@book{ziegler1995LecturesPolytopes,
  title = {Lectures on Polytopes},
  author = {Ziegler, Günter M.},
  year = {1995},
  edition = {1},
  publisher = {Springer},
  doi = {10.1007/978-1-4613-8431-1}
}

\begin{appendices}

\section{Number of Parameters in the Input}

\begin{example}\label{example::pMustBeConstant}
    In this example, we demonstrate that, if~$p$ is not considered to be constant, we can construct a family of parametric minimization problems where the cardinality of every $(1+\varepsilon)$-approximation set grows super-polynomially in the instance size.

    For even~$p$, let~$\Pi$ be a $p$-parametric minimization integer program such that
    \begin{itemize}
        \item $\Lambda=\mathbb{R}^p_{\geq0}$,
        \item $X=\left\{x\in{\{0,1\}}^p:\sum_{i=1}^p x_i=\frac{p}{2}\right\}$, and
        \item $F_\lambda(x)=\sum_{i=1}^p \lambda_i x_i$.
    \end{itemize}
    Clearly, $\langle I\rangle\in \bigO{\left(\poly\langle p\rangle\right)}$.
    Furthermore, for all~$x\in X$ and all~$\lambda\in\Lambda$, it holds that~$F_\lambda(x)\geq0$.

    For an arbitrary~$x\in X$, let~$\lambda\in\mathbb{R}_{\geq0}^p$ be the parameter vector where $\lambda_i=1-x_i$ for $i=1,\ldots,p$.
    Since either $x_i=0$ or $\lambda_i=0$, it holds that $F_\lambda(x)=F^*(\lambda)=0$.
    Conversely, for every~$x'\in X\setminus\{x\}$, there is at least one $i\in\{1,\ldots,p\}$ such that $\lambda_i=x'_i=1$ and, thus, $F_\lambda(x') \geq 1$.
    Consequently, $x$ is the only solution that can $(1+\varepsilon)$-approximate~$\Pi(\lambda)$.
    Therefore, the $(1+\varepsilon)$-approximation set is~$X$ itself.

    Let~$\mathsf{bin}$ denote the binomial coefficient.
    By construction, $\left|X\right|=\mathsf{bin}\binom{p}{p/2}$.
    It is well known that $\sum_{i=0}^p \mathsf{bin}\binom{p}{i}=2^p$ and that $\mathsf{bin}\binom{p}{p/2}$ maximizes $\mathsf{bin}\binom{p}{i}$, so $(p+1)\cdot \binom{p}{p/2}\geq 2^p$.
    Thus, we can lower bound $\left|X\right|$ in $\Omega(p^{-1}\cdot 2^p)$.
    With increasing~$p$, the encoding length of~$I$ grows polynomially in $p$, but the cardinality of every $(1+\varepsilon)$-approximation set grows exponentially in~$p$.

    Without proof, we remark that this also holds if we use the linear relaxation of~$\Pi$, where the set of feasible solutions is $\left\{x\in{[0,1]}^p:\sum_{i=1}^p x_i=\frac{p}{2}\right\}$.%
    \footnote{See Section~\ref{sec::maxNotApx}, where we argue that it is sufficient to consider only certain vertices of a parametric linear program.}
\end{example}

\section{Supplementary Lemmas for Section~\ref{sec::maxNotApx}}

\begin{lemma}\label{lemma::maxNotApx::setsEqual}
    For every $\xi\in[0,1]$, the sets
\[
\left\{\left({c^1}^\T x, {c^2}^\T x\right): A x \geq\xi b, x\in\mathbb{R}^n \right\}
\]
and
\[
\left\{\xi \left({c^1}^\T x, {c^2}^\T x\right): A x \geq b, x\in\mathbb{R}^n \right\}
\]
from Section~\ref{sec::maxNotApx} are equal.
\end{lemma}
\begin{proof}
    Note that, by construction in Section~\ref{sec::maxNotApx}, both sets are bounded.
    Then, for~$\xi>0$, the statement of the theorem follows from the following chain of equalities:
    \begin{align*}
        & \left\{\left({c^1}^\T x, {c^2}^\T x\right): A x \geq\xi b,\, x\in\mathbb{R}^n \right\} \\
        = & \left\{\left({c^1}^\T x, {c^2}^\T x\right): A (\xi^{-1}x) \geq b,\, x\in\mathbb{R}^n \right\} \\
        = & \left\{\left({c^1}^\T (\xi y), {c^2}^\T (\xi y)\right): A y \geq b,\, \xi y= x,\, x\in\mathbb{R}^n,\, y\in\mathbb{R}^n \right\} \\
        = & \left\{\xi \left({c^1}^\T y, {c^2}^\T y\right): A y \geq b,\, y\geq0 \right\}.
    \end{align*}
    As a consequence, for $\xi \rightarrow 0$, all points in both sets approach zero.
    Since both sets are convex and bounded, they must then both be equal to~$\{0\}$ for~$\xi=0$.
\end{proof}

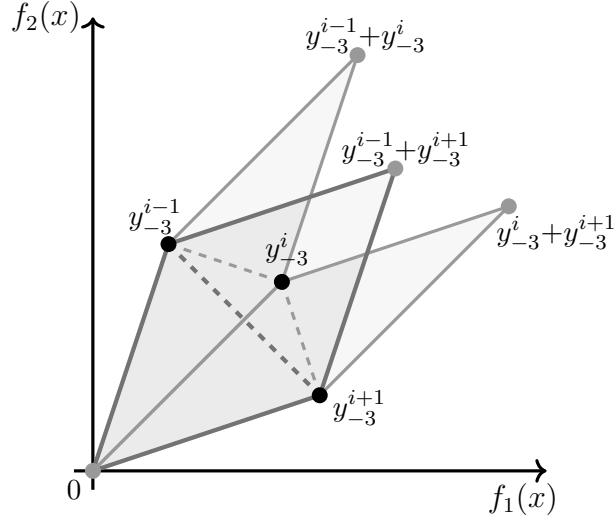
\begin{figure}
    \centering
    \begin{tikzpicture}[scale=1.0,
    smalledge/.style={draw=black!40, very thick},
    bigedge/.style={draw=black!55, ultra thick},
    smalldashed/.style={draw=black!40, very thick, dashed},
    bigdashed/.style={draw=black!55, ultra thick, dashed},
    lightfill/.style={draw=none, fill=black!3},
    hardfill/.style={draw=none, fill=black!8},
    verti/.style={draw=none, fill=black!40},
    vertog/.style={draw=none, fill=black!100},
]
	\draw[very thick,->] (-0.25,0) to node[pos=0.95, below] {$f_1(x)$} (6, 0);
	\draw[very thick,->] (0,-0.25)to node[pos=1.0, left] {$f_2(x)$} (0, 6.0);

    \coordinate (origin) at (0,0);
    \coordinate (yiminus) at (1,3);
    \coordinate (yi) at (2.5, 2.5);
    \coordinate (yiplus) at (3,1);
    \coordinate (bminus) at ($(yiminus)+(yi)$);
    \coordinate (b) at ($(yiminus)+(yiplus)$);
    \coordinate (bplus) at ($(yi)+(yiplus)$);

    \draw [name path=mleft] (yiminus) -- (b);
    \draw [name path=mright] (bminus) -- (yi);
    \path [name intersections={of=mleft and mright,by=interminus}];
    \draw [name path=pleft] (yiplus) -- (b);
    \draw [name path=pright] (bplus) -- (yi);
    \path [name intersections={of=pleft and pright,by=interplus}];

    \draw[lightfill] (yiminus) to (bminus) to (interminus) to cycle;
    \draw[lightfill] (yiplus) to (bplus) to (interplus) to cycle;
    \draw[lightfill] (yi) to (interminus) to (b) to (interplus) to cycle;
    \draw[hardfill] (origin) to (yiminus) to (interminus) to (yi) to (interplus) to (yiplus) to cycle;

    \draw[smalledge] (origin) to (yiminus) to (bminus) to (yi) to cycle;
    \draw[smalledge] (origin) to (yiplus) to (bplus) to (yi) to cycle;
    \draw[bigedge] (origin) to (yiminus) to (b) to (yiplus) to cycle;

    \draw[smalldashed] (yiminus) to (yi);
    \draw[smalldashed] (yiplus) to (yi);
    \draw[bigdashed] (yiminus) to (yiplus);

  \fill[verti] (origin) circle (3pt);
  \node[xshift=-7pt, yshift=-7pt] at (origin) {\small$0$};

  \fill[vertog] (yiminus) circle (3pt);
  \node[xshift=-5pt, yshift=10pt] at (yiminus) {\small$y_{-3}^{i-1}$};

  \fill[vertog] (yi) circle (3pt);
  \node[xshift=2pt, yshift=12pt] at (yi) {\small$y_{-3}^i$};

  \fill[vertog] (yiplus) circle (3pt);
  \node[xshift=15pt, yshift=-5pt] at (yiplus) {\small$y_{-3}^{i+1}$};

  \fill[verti] (bminus) circle (3pt);
  \node[xshift=3pt, yshift=7pt] at (bminus) {\small\thinmuskip=0mu\medmuskip=0mu\thickmuskip=2mu$y_{-3}^{i-1}+y_{-3}^i$};

  \fill[verti] (b) circle (3pt);
  \node[xshift=4pt, yshift=7pt] at (b) {\small\thinmuskip=0mu\medmuskip=0mu\thickmuskip=2mu$y_{-3}^{i-1}+y_{-3}^{i+1}$};

  \fill[verti] (bplus) circle (3pt);
  \node[xshift=18pt, yshift=-10pt] at (bplus) {\small\thinmuskip=0mu\medmuskip=0mu\thickmuskip=2mu$y_{-3}^i+y_{-3}^{i+1}$};

\end{tikzpicture}
    \caption{%
    Graphical interpretation of the proof of Lemma~\ref{lemma::maxNotApx::determinants}:
    The parallelepipeds spanned by $\left(y_{-3}^{i-1},y_{-3}^i\right)$, and $\left(y_{-3}^i,y^{i+1}\right)$ are drawn with thin edges, and the parallelepiped spanned by $\left(y_{-3}^{i-1},y_{-3}^{i+1}\right)$ is drawn with thick edges.
    The dashed lines indicate the separation of the parallelepipeds into two halves each.%
    \label{figure::maxNotApx::parallelepiped}
    }
\end{figure}

\begin{lemma}\label{lemma::maxNotApx::determinants}
    In Section~\ref{sec::maxNotApx}, choose any $i,j,\ell$ such that $1\leq i < j < \ell \leq N$.
    Then, it holds that
    \[
        \frac{\det(y^j,y^i)+\det(y^\ell,y^j) }{\det(y^\ell,y^i)} > 1.
    \]
\end{lemma}
\begin{proof}
    In this proof, for every $i=1,\ldots,N$, let~$y_{-3}^i$ denote the projection of a point to its first two components, i.e., $y_{-3}^i\coloneqq(y_1^i,y_2^i)^\T$.
    As a shorthand, for finitely many points $w^1,\ldots, w^t\in\mathbb{R}^2$, we denote the volume of the polyhedron $\conv{\left(w^1,\ldots,w^t\right)}$ by $\vol{\left(w^1,\ldots,w^t\right)}$.

    We now use the geometric properties of the determinant to prove Lemma~\ref{lemma::maxNotApx::determinants}.
    Figure~\ref{figure::maxNotApx::parallelepiped} visualizes the idea of the proof.
    The determinant $\det(y^j,y^i)$ is equivalent to the volume of the parallelepiped spanned by the points~$y_{-3}^i)$ and~$y_{-3}^j$, i.e., $\det(y^j,y^i)=\vol{\left(0,y_{-3}^i,y_{-3}^j,y_{-3}^i+y_{-3}^j \right)}$~\cite{valenza1993LinearAlgebra}.
    Then, we get $\frac{1}{2}\det(y^j,y^i)=\vol{\left(0,y_{-3}^i,y_{-3}^j\right)}$ by separating the parallelepiped into two halves.
    This can be done analogously for the pair~$y^\ell$ and~$y^i$, and for the pair~$y^\ell$ and~$y^j$.
    Therefore,
    \[
        \frac{1}{2}\left(\det(y^j,y^i)+\det(y^\ell,y^j)\right)=\vol\left(0,y_{-3}^i,y_{-3}^j\right) + \vol\left(0,y_{-3}^j,y_{-3}^\ell\right).
    \]
    By construction, $y^i$ and~$y^\ell$ are on different sides of the $0$-$y_{-3}^j$-line, and, thus, the two polyhedra $\conv{\left(0,y_{-3}^i,y_{-3}^j\right)}$ and $\conv{\left(0,y_{-3}^j,y_{-3}^\ell\right)}$ only intersect in their boundaries.
    This leads to
    \[
        \vol{\left(0,y_{-3}^i,y_{-3}^j\right)} + \vol{\left(0,y_{-3}^j,y_{-3}^\ell\right)}=\vol{\left(0, y_{-3}^i,y_{-3}^j,y_{-3}^\ell\right)}.
    \]

    By construction, $y_{-3}^i$, $y_{-3}^j$, and~$y_{-3}^\ell$ are all vertices of this polyhedron and, thus,
    \[
    \conv{\left(0, y_{-3}^i,y_{-3}^\ell\right)}\subset\conv{\left(0, y_{-3}^i,y_{-3}^j,y_{-3}^\ell\right)}.
    \]
    This leads to
    \[
    \frac{1}{2}\det\left(y^\ell,y^i\right)=\vol{\left(0, y_{-3}^i,y_{-3}^\ell\right)}<\vol{\left(0, y_{-3}^i,y_{-3}^j,y_{-3}^\ell\right)}.
    \]

    Therefore,
    \[
    \frac{1}{2}\left(\det(y^j,y^i)+\det(y^\ell,y^j)\right) > \frac{1}{2}\det\left(y^\ell,y^i\right),
    \]
    which can be transformed into the desired statement.
\end{proof}

\end{appendices}

\end{document}